\documentclass[final]{siamltex}
\pdfoutput=1
\usepackage{amsmath}
\usepackage{epsfig}
\usepackage{graphicx}
\usepackage{amssymb}
\newtheorem{remark}{Remark}[section]
\def\lam{{\lambda}}

\def\Ome{{\Omega}}

\def\del{{\delta}}
\def\Del{{\Delta}}
\def\nab{{\nabla}}
\def\vepsi{{\varepsilon}}
\def\p{{\partial}}
\def\reff#1{\eqref{#1}}
\def\norm#1#2{\Vert\,#1\,\Vert_{#2}}

\def\hnorm#1#2{\vert\,#1\,\vert_{#2}}

\def\vepsi{\varepsilon}

\def\Div{{\rm div }}
\def\cA{{\mathcal A}}
\def\cB{{\mathcal B}}
\def\cH{{\mathcal H}}
\def\cT{{\mathcal T}}
\def\cM{{\mathcal M}}
\def\cL{{\mathcal L}}

\def\Tc{{\cal T}}

\def\p{{\partial}}

\def\nab{\nabla}
\def\Ome{\Omega}
\def\lam{\lambda}
\def\Del{\Delta}

\def\bbf{\mathbf{f}}
\def\bu{\mathbf{u}}
\def\bv{\mathbf{v}}
\def\bw{\mathbf{w}}

\def\bg{\mathbf{g}}
\def\bn{\mathbf{n}}

\def\bI{\mathbf{I}}
\def\bV{\mathbf{V}}

\def\bR{\mathbf{R}}
\def\bU{\mathbf{U}}
\def\bV{\mathbf{V}}

\def\bY{\mathcal{Y}}

\begin{document}
\title{On $p$-harmonic map heat flows for {$1\leq p< \infty$} and their 
finite element approximations}

\author{
John W. Barrett\thanks{Department of Mathematics, Imperial College London,
London, SW7 2AZ, U.K. (j.barrett@ic.ac.uk).}
\and
Xiaobing Feng\thanks{Department of Mathematics, The University of
Tennessee, Knoxville, TN 37996, U.S.A. (xfeng@math.utk.edu). The work of this 
author was partially supported by the NSF grant DMS-0410266.}
\and
Andreas Prohl\thanks{Mathematisches Institut, Universit\"at T\"ubingen,
Auf der Morgenstelle 10, D-72076 T\"ubingen, Germany.
(prohl@na.uni-tuebingen.de)}
}

\maketitle

\begin{abstract}
Motivated by emerging applications from imaging processing, the
heat flow of a generalized $p$-harmonic map into spheres is
studied for the whole spectrum, $1\leq p<\infty$, in a unified
framework. The existence of global weak solutions is established
for the flow using the energy method together with a
regularization and a penalization technique. In particular, 
a $BV$-solution concept is introduced and the existence 
of such a solution is proved for the $1$-harmonic map heat flow.
The main idea used to develop such a theory is to exploit the 
properties of measures of the forms $\cA\cdot\nab\bv$ 
and $\cA\wedge\nab\bv$; which pair a divergence-$L^1$, or a 
divergence-measure, tensor field $\cA$, and a
$BV$-vector field $\bv$. Based on these analytical results, a practical 
fully discrete finite element method is then proposed for approximating 
weak solutions of the $p$-harmonic map heat flow, and the convergence 
of the proposed numerical method is also established. 
\end{abstract}



\pagestyle{myheadings}
\thispagestyle{plain}
\markboth{JOHN W. BARRETT AND XIAOBING FENG AND ANDREAS PROHL}
{HEAT FLOW OF $p$-HARMONIC MAPS}

\section{Introduction}\label{sec-1}

Let $\Omega \subset \mathbf{R}^m$  be a bounded domain with smooth
boundary $\p\Ome$, and $S^{n-1}$ denote the unit sphere in $\mathbf{R}^{n}$.
A map ${\bu}\in C^1(\Omega, S^{n-1})$ is called a {\em $p$-harmonic map}
if it is a critical point of the following $p$-energy
\begin{equation}\label{e1.1}
E_p({\bv}) := \frac{1}{p} \int_\Omega \hnorm{\nabla {\bv}}{}^p\,
{\rm d}{x}, \qquad 1\leq p < \infty.
\end{equation}

It is well-known, \cite{CHH,FR1,M1}, that the Euler-Lagrange equation
of the $p$-energy is
\begin{equation}\label{e1.2}
-\Del_p\bu = |\nab \bu|^p \bu ,
\end{equation}
where
\begin{equation}\label{e1.2a}
\Del_p\bu:=\Div(|\nab \bu|^{p-2} \nab \bu).
\end{equation}
Note that $\Del_p$ is often called the $p$-Laplacian. It is easy
to see that equation \reff{e1.2} is a {\em degenerate elliptic}
equation for $p>2$, and a {\em singular elliptic} equation for
$1\leq p< 2$. These degeneracy and singular characteristics
both disappear when $p=2$.

We call a map ${\bu}\in W^{1,p}(\Omega, S^{n-1})$ a {\em weakly $p$-harmonic map}
if ${\bu}$ satisfies equation \reff{e1.2} in the distributional sense.
Here $W^{1,p}(\Omega, S^{n-1})$ denotes the Sobolev space
\[
W^{1,p}(\Omega, S^{n-1}):=\left\{ \bu\in W^{1,p}(\Omega,\mathbf{R}^{n});\,
\bu(x)\in S^{n-1}\mbox{ for a.e. } x\in \Ome \right\}.
\]
One well-known method for looking for a weakly $p$-harmonic map is
the homotopy method (or the gradient descent method), which then
leads to considering the following gradient flow (or heat flow)
for the $p$-energy functional $E_p$
\begin{alignat}{2}\label{e1.3}
\bu_t - \Del_p \bu &= |\nabla {\bu}|^p {\bu}
&&\qquad\mbox{in }\Ome_T:=\Ome\times (0,T), \\
|\bu| &=1 &&\qquad \mbox{in }\Ome_T, \label{e1.4}
\end{alignat}
complemented with some given boundary and initial conditions.
Clearly, equation \reff{e1.3} is a {\em degenerate parabolic}
equation for $p>2$, and a {\em singular parabolic} equation for
$1\leq p< 2$. Again, these degeneracy and singular characteristics both
disappear when $p=2$.

We remark that $p$-harmonic maps and weakly $p$-harmonic maps
between two Riemannian manifolds $(M, g)$ and $(N, h)$, and their
heat flows can be defined in the same fashion (cf. \cite{FR1,S1}).
In this paper, we shall only consider the case $M=\Ome$ and $ N=S^{n-1}$,
which is sufficient for the applications that we are interested in.

The $p$-harmonic map and its heat flow, in particular, the
harmonic map and the harmonic map heat flow ($p=2$), have been
extensively studied in the past twenty years for $1<p<\infty$ (cf.
\cite{C,CHH,CS,CD,ES,FR1,FR2,HL,Hu1,Hu2,Hu3,Liu1,Liu2,M1,M2,S1,S2} 
for $1<p<\infty$; \cite{GMS1,GMS2} for
$p=1$). In the case when the target manifold is a sphere, the
existence of a global weak solution for the harmonic flow was first
proved by Chen in \cite{C} using a penalization technique. The
result and the penalization technique were extended to the
$p$-harmonic flow for $p>2$ by Chen, Hong and Hungerb\"uhler in
\cite{CHH}. The $p$-harmonic flow for $1<p<2$ was solved by Misawa
in \cite{M2} using a time discretization technique (the method of
Rothe) proposed in \cite{Hu2}, and by Liu in \cite{Liu2} using a
penalization technique similar to that of \cite{CHH}. The
$p$-harmonic flow ($1<p<\infty$) from a unit ball in
$\mathbf{R}^m$ into $S^1\subset \mathbf{R}^2$ was studied by
Courilleau and Demengel in \cite{CD}. In the case of general
target manifolds the $p$-harmonic flow ($1<p<\infty$) with small
initial data was treated by Fardoun and Regbaoui in \cite{FR2},
and the conformal case of the $p$-harmonic flow was considered 
by Hungerb\"uhler in \cite{Hu1a}. 
Nonuniqueness of the $p$-harmonic flow was addressed in
\cite{Co,Hu1}. Recently, Giga {\em et al.} \cite{GKY04} 
proved existence of strong local solutions for $1$-harmonic
map heat flow using nonlinear semigroup theory.
Besides the great amount of mathematical interests in the
$p$-harmonic map and the $p$-harmonic flow, research on these problems
has been strongly motivated by applications of the harmonic map
and its heat flow in liquid crystals and micromagnetism, we refer
to \cite{BBH,B,EK,GH,LL,L,Virga} and the references therein for
detailed expositions in this direction.

Another reason, which is the main motivation of this paper, for
studying the $p$-harmonic map and its heat flow, in particular,
for $1\leq p <2$, arises from their emerging and intriguing
applications to image processing for denoising color images (cf.
\cite{TSC,VO}). We recall that a color image is often expressed
by the RGB color system, in which a vector ${\bf I}(x)=(r(x),
g(x), b(x))$ for each pixel $x=(x_1,x_2)$ is used to represent the
intensity of the three primary colors (red, green, and blue). The
chromaticity and brightness of a color image are deduced from the
RGB system by decomposing ${\bf I}(x)$ into two components
\begin{eqnarray*}
\eta(x) &:=& \hnorm{{\bf I}(x)}{}, \qquad\mbox{(brightness)}\\
{\bg} &:=& \frac{ {\bf I}(x)}{\hnorm{{\bf I}(x)}{}}, \qquad\mbox{(chromaticity)}
\end{eqnarray*}
where $|\bI(x)|$ stands for the Euclidean norm of $\bI(x)$.
By definition, the chromaticity must lie on the unit sphere $S^2$.
One of the main benefits of the chromaticity and brightness decomposition
is that it allows one to denoise $\eta$ and ${\bg}$ separately using different
methods. For example, one may denoise $\eta$ by the well-known total
variation (TV) model of Rudin-Osher-Fatemi \cite{ROF} (also see \cite{FP}),
but denoise ${\bg}$ by another model. One such model is to define the
recovered chromaticity ${\bu}$ as a (generalized) $p$-harmonic
map \cite{TSC,VO}
\begin{equation}\label{e1.6}
\bu = \underset{ \bv \in W^{1,p}(\Omega;S^2)} {\rm argmin}
J_{p,\lam}({\bv})  \qquad\mbox{for } p\geq 1,
\end{equation}
where
\begin{eqnarray}\label{e1.7}
J_{p,\lam}({\bv})&:=& E_p(\bv)+\frac{\lam}2\int_\Ome |\bv-\bg|^2\, dx
\qquad\mbox{for } \lam > 0.
\end{eqnarray}
In particular, the cases $1 \leq p < 2$ are the most important and
interesting since the recovered images keep geometric information such
as edges and corners of the noisy color images. We shall call \reff{e1.6}
the {\em $p$-harmonic model} for color image denoising. We also remark that the
second term on the right-hand side of \reff{e1.7} is often called a fidelity
term. As in the TV model \cite{ROF,FP}, the parameter $\lam$ controls
the trade-off between goodness of fit-to-the-data and variability in $\bu$.

Again, to find a solution for the $p$-harmonic map model \reff{e1.6}, 
we consider the gradient flow (heat flow) for the energy
functional $J_{p,\lam}$, which is given by
\begin{alignat}{2}\label{e1.8}
\bu_t - \Del_p \bu + \lam (\bu-\bg) &= \mu_{p,\lam} {\bu}
&&\qquad\mbox{in }\Ome_T, \\
|\bu| &=1 &&\qquad \mbox{in }\Ome_T, \label{e1.9} \\
\cB_p\bn &=0 &&\qquad\mbox{on } \p\Ome_T:=\p\Ome\times (0,T),
\label{e1.10}\\
\bu &=\bu_0 &&\qquad\mbox{on } \Ome\times \{t=0\},  \label{e1.11}
\end{alignat}
where
\begin{equation}\label{e1.11a}
\cB_p:=|\nab \bu|^{p-2} \nab\bu,\qquad \mu_{p,\lam}:=|\nabla {\bu}|^p + \lam\,(1-\bu\cdot\bg).
\end{equation}

The goals of this paper are twofold. Firstly, we shall present a general theory
of weak solutions for the parabolic system
\reff{e1.8}-\reff{e1.11} for the whole spectrum $1\leq p<\infty$.
To the best of our knowledge, there is no theory known in the
literature for the $1$-harmonic map and its heat flow, which on
the other hand is the most important (and most difficult) case for
the color image denoising application. So our theory and result
fill this void. Furthermore, our theory handles the $p$-harmonic
map heat flow \reff{e1.8}-\reff{e1.11} for all $1\leq p<\infty$ in
a unified fashion, rather than treating the system separately for
different values of $p$ and using different methods (cf. 
\cite{C,CHH,CS,CD,ES,FR1,FR2,HL,Hu2,Hu3,M2,S1,S2}). Secondly,
based on the above theoretical work, we also develop and analyze
a practical fully discrete finite element method for approximating
the solutions of the $p$-harmonic map heat flow
\reff{e1.8}-\reff{e1.11}.

We now highlight the main ideas and key steps of our approach.
Notice that there are two nonlinear terms in the $p$-harmonic
flow: the $p$-Laplace term and the right-hand side due to the
nonconvex constraint $|\bu|=1$, so the main difficulties are how
to handle these two terms and how to pass to the limit in these
two terms when a compactness argument is employed.  To handle the
degeneracy of the $p$-Laplace term, we approximate the $p$-energy
$E_p(\bv)$ by the following regularized energy
\begin{eqnarray}\label{e1.12}
E_p^\vepsi(\bv)&:=&\frac{b_p(\vepsi)}{2}\int_\Ome |\nab \bv|^2\, dx
+\frac{1}{p} \int_\Ome |\nab \bv|_\vepsi^p\, dx \\
&=&\int_\Ome \Bigl\{ \frac{b_p(\vepsi)}{2} |\nab \bv|^2 +\frac{1}{p}
\bigl[\,|\nab \bv|^2+a_p(\vepsi)^2\,\bigr]^{\frac{p}2}\,\Bigr\} dx,
\nonumber
\end{eqnarray}
where $\vepsi>0$ and
\begin{equation}\label{e1.12a}
a_p(\vepsi):=\left\{ \begin{array}{ll}
                    0 &\quad\mbox{if } 2\leq p<\infty,\\
                    \vepsi &\quad\mbox{if }  1\leq p< 2,
                      \end{array} \right.
\qquad
b_p(\vepsi):=\vepsi^\alpha \quad\mbox{for } 1\leq p< \infty
\end{equation}
for some $\alpha>0$. Here and in the rest of this paper we adopt the
shorthand notation
\begin{equation}\label{e1.13}
|\nab \bv|_\vepsi:=\sqrt{|\nab \bv|^2+a_p(\vepsi)^2}.
\end{equation}

To handle the nonconvex constraint $|\bu|=1$, we approximate it by
the well-known Ginzburg-Landau penalization \cite{BBH}, that is,
we abandon the exact constraint, but enforce it approximately by
adding a penalization term to the regularized $p$-energy
$E_p^\vepsi$. To this end, we replace the energy $E_p^\vepsi$ by
\begin{equation}\label{e1.14}
E_p^{\vepsi,\delta}(\bv):= E_p^\vepsi(\bv) + L^\delta(\bv)\qquad\mbox{for }
\vepsi,\delta>0,
\end{equation}
where
\begin{eqnarray}\label{e1.15}
L^\delta(\bv):=\frac{1}{\delta}\int_\Ome F(\bv)\, dx, \qquad
F(\bv):=\frac{1}{4} \bigl(\,|\bv|^2-1\, \bigr)^2 \qquad
\forall\delta>0,\,\bv\in \mathbf{R}^n.
\end{eqnarray}
So the idea is, as $\delta$ gets smaller and smaller, the energy
functional $E_p^{\vepsi,\delta}$ becomes more and more
favorable for maps $\bu$ which take values close to the unit sphere $S^{n-1}$.

Consequently, the regularized model for the $p$-harmonic
model \reff{e1.6} (with general $m$ and $n$) reads
\begin{equation}\label{e1.16}
\bu^{\vepsi,\del}=\underset{\bv \in W^{1,p}(\Omega;\mathbf{R}^{n})}
{\rm argmin} J^{\vepsi,\del}_{p,\lam}({\bv})  \qquad\mbox{for } p\geq 1,
\end{equation}
where
\begin{eqnarray}\label{e1.17}
J^{\vepsi,\del}_{p,\lam}({\bv})&:=& E^{\vepsi,\del}_p(\bv)
+\frac{\lam}2\int_\Ome |\bv-\bg|^2\, dx .
\end{eqnarray}
In addition, the gradient flow for the regularized energy functional
$J^{\vepsi,\del}_{p,\lam}$ is given by
\begin{alignat}{2}\label{e1.18}
\bu^{\vepsi,\del}_t - \Del_p^\vepsi \bu^{\vepsi,\del}
+\frac{1}{\del}\bigl(|\bu^{\vepsi,\del}|^2-1\bigr)
\bu^{\vepsi,\del} + \lam (\bu^{\vepsi,\del}-\bg)
&=0 &&\qquad\mbox{in }\Ome_T, \\
\cB_p^{\vepsi,\del} \bn &=0 &&\qquad\mbox{on } \p\Ome_T,
\label{e1.19}\\
\bu^{\vepsi,\del} &=\bu_0
&&\qquad\mbox{on } \Ome\times \{t=0\}, \label{e1.20}
\end{alignat}
which is an approximation to the original flow
\reff{e1.8}-\reff{e1.11}, where
\begin{equation}\label{e1.21}
\cB_p^{\vepsi,\delta}:= b_p(\vepsi)\nab \bu^{\vepsi,\delta} +
|\nab\bu^{\vepsi,\delta}|_\vepsi^{p-2}\nab \bu^{\vepsi,\delta},\qquad
\Del_p^\vepsi\bu^{\vepsi,\delta}:=\Div \cB_p^{\vepsi,\delta}.
\end{equation}

After having introduced the regularized flow \reff{e1.18}-\reff{e1.20},
the next step is to analyze this regularized flow, in particular, to
derive uniform (in $\vepsi$ and $\del$) a priori estimates. Finally, we
pass to the limit in \reff{e1.18}-\reff{e1.20}, first letting
$\del\rightarrow 0$ and then setting
$\vepsi\rightarrow 0$. As pointed out earlier, the main difficulty
here is passing to the limit in two nonlinear terms on the left-hand
side of \reff{e1.18}. For $1<p<\infty$, this will be done
using a compactness technique and exploiting the symmetries of
the unit sphere $S^{n-1}$, as done in \cite{C,CHH,CS,FR1,Hu1,M2,S1}.
However, for $p=1$, since $L^1(\Ome)$ is not a reflexive Banach space, instead
of working in the Sobolev space $W^{1,1}(\Ome)$, we are forced to work
in $BV(\Ome)$, the space of functions of bounded variation, since
solutions of the original $1$-harmonic map heat flow
\reff{e1.8}-\reff{e1.11} only belong to  $L^\infty((0,T);[BV(\Ome)]^n)$
in general. This lack of regularity makes the analysis for $p=1$ much more
delicate than that for $1<p<\infty$.

We note also that the regularized flow \reff{e1.18}-\reff{e1.20}
not only plays an important role for proving existence of weak
solutions for the flow \reff{e1.8}-\reff{e1.11}, but also provides
a practical and convenient formulation for approximating the
solutions. The second goal of this paper is to develop a
practical fully discrete finite element method for approximating
solutions of the regularized flow \reff{e1.18}-\reff{e1.20};
and hence, approximating solutions of the original $p$-harmonic flow
\reff{e1.8}-\reff{e1.11} via the regularized flow. It is well-known 
that explicitly enforcing the nonconvex constraint $|\bu|=1$
at the discrete level is hard to achieve. The penalization used in
\reff{e1.18} allows one to get around this numerical difficulty at
the expense of introducing an additional scale $\del$. As
expected, the numerical difficulty now is to control the
dependence of the regularized solutions on $\del$ and to establish
scaling laws which relate the numerical scales (spatial and
temporal mesh sizes) to the penalization scale $\del$ for both
stability and accuracy concerns. We refer to \cite{BFP1} and the
reference therein for more discussions in this direction and
discussions on a related problem arising from liquid crystal
applications. Borrowing a terminology from the phase transition of
materials science, the regularized flow \reff{e1.18}-\reff{e1.20}
may be regarded as a ``diffuse interface" model for the original
``sharp interface" model \reff{e1.8}-\reff{e1.11}, and the
``diffuse interface" is represented by the region $\{
|\bu^{\vepsi,\del}|< 1-\del\}$.

The remaining part of the paper is organized as follows.  In
Section \ref{sec-2} we collect some known results and facts, which
will be used in the later sections. In Section \ref{sec-3} we
present a complete analysis for the regularized flow
\reff{e1.18}-\reff{e1.20}, which includes proving its
well-posedness, an energy law, a maximum principle, and uniform
(in $\vepsi$ and $\del$) a priori estimates. As expected, these
uniform a priori estimates serve as the basis for carrying out the
energy method and the compactness arguments in the later sections.
In Section \ref{sec-4} we pass to the limit in
\reff{e1.18}-\reff{e1.20} as $\del\rightarrow 0$ for each {\em
fixed} $\vepsi>0$. As in \cite{CS,CHH,Hu1,M2}, the main idea is to
exploit the monotonicity of the operator $\Del^\vepsi_p$ and the
symmetries of the unit sphere $S^{n-1}$. Sections
\ref{sec-5}-\ref{sec-6} are devoted to passing to the limit as
$\vepsi\rightarrow 0$ in the $\vepsi$-dependent limiting system
obtained in Section \ref{sec-4}. For $1<p<\infty$, this will be
done by following the idea of Section \ref{sec-4}. However, for
$p=1$ the analysis becomes much more delicate because the
nonreflexivity of $W^{1,1}(\Ome)$ forces us to work in the
$BV(\Ome)$ space. The main idea used to develop a $BV$
solution concept is to exploit the properties of measures 
of the forms $\cA\cdot\nab\bv$ and $\cA\wedge\nab\bv$; which 
pair a divergence-$L^1$, or a divergence-measure, tensor 
field $\cA$, and a
$BV$-vector field $\bv$. Finally, based on the theoretical 
results of Sections \ref{sec-3}-\ref{sec-6}, in Section \ref{sec-7} we 
propose and analyze a practical fully discrete finite element method 
for approximating solutions of the $p$-harmonic flow \reff{e1.8}-\reff{e1.11} 
via the regularized flow \reff{e1.18}-\reff{e1.20}. It is proved that
the proposed numerical scheme satisfies a discrete energy inequality, 
which mimics the
differential energy inequality, and this leads to uniform (in
$\vepsi$ and $\del$) a priori estimates and the convergence of the
numerical approximations to the solutions of the flow
\reff{e1.8}-\reff{e1.11} as the spatial and temporal mesh sizes,
and the parameters $\vepsi$ and $\del$, all tend to zero.

\section{Preliminaries}\label{sec-2}

The standard notation for spaces is adopted in this paper (cf.
\cite{adams75,AFP}). For example,
$W^{k,p}(\Ome)$, $k\geq 0$ integer and $1\leq p<\infty$, denotes
the Sobolev spaces over the domain $\Ome$ and $\|\cdot\|_{W^{k,p}}$
denotes its norm.
$W^{0,p}(\Ome)=L^p(\Ome)$ and $W^{k,2}(\Ome)=H^k(\Ome)$ are also used.
$L^q((0,T); W^{k,p}(\Ome, \mathbf{R}^n))$ denotes the space of 
vector-valued functions (or maps), whose $W^{k,p}(\Ome)$-norm is
$L^q$-integrable as a function of $t$ over the interval $(0,T)$,
and $\|\cdot\|_{L^q(W^{k,p})}:=(\int_0^T\|\cdot\|_{W^{k,p}}^q dt)^{\frac{1}{q}}$,
for $q \in [1,\infty)$,
denotes its norm; and with the standard modification for $q= \infty$.  
We use $(\cdot,\cdot)$ to denote the
standard inner product in $L^2(\Ome)$, and $\langle\cdot,\cdot\rangle$ 
to denote a generic dual product between elements of a Banach space $X$ 
and its dual space $X'$.

In addition, $\mathcal{M}(\Ome)$ (resp. $[\mathcal{M}(\Ome)]^n$) denotes
the space of real-valued (resp. $\mathbf{R}^n$-valued) finite
Radon measures on $\Ome$. Recall that $\mathcal{M}(\Ome)$ is the 
dual space of $C_0(\Ome)$ (cf. \cite{AFP}). For a positive non-Lebesgue 
measure $\mu\in \mathcal{M}(\Ome)$, $L^p(\Ome, \mu)$ is used to denote the space
of $L^p$-integrable functions with respect to the measure $\mu$,
and for $f\in L^1(\Ome, \mu)$, the measure $f\mu$ is defined by
\[
f\mu(A):=\int_A\, f\,d\mu\qquad\mbox{for any Borel set } A\subset \Ome.
\]
For any $\mu\in \mathcal{M}(\Ome)$, $\mu=\mu^a+\mu^s$ denotes its
Radon-Nikod\'ym decomposition, where $\mu^a$ and $\mu^s$ respectively
denotes the absolute continuous part and the singular part of $\mu$ 
with respect to the Lebesgue measure $\cL^n$. 

In addition, $BV(\Ome)$ is used to denote the space of functions
of bounded variation. Recall that a function $u\in L^1(\Ome)$ is
called a function of {\em bounded variation} if all of its first order
partial derivatives (in the distributional sense) are measures with
finite total variations in $\Ome$. Hence, the gradient of
such a function $u$, denoted by $D u$, is a bounded
vector-valued measure, with the finite total variation
\begin{equation}\label{e2.1}
|D u|\equiv |Du|(\Ome):
=\sup\Bigl\{\, \int_\Ome -u\, \Div{\,\mathbf v}\, dx\, ; \,\,
{\mathbf v}\in [C^1_0(\Ome)]^n,\, \norm{{\mathbf v}}{L^\infty} \leq 1 \,
\Bigr\}\, .
\end{equation}
$BV(\Ome)$ is known to be a Banach space endowed with the norm
\begin{equation}\label{e2.2}
\norm{u}{BV}:= \norm{u}{L^1} + |D u|.
\end{equation}
For any $u\in BV(\Ome)$, $(Du)^a$ and $|D u|^a$ are used to denote
respectively the absolute continuous part of $Du$ and $|D u|$ 
with respect to the Lebesgue measure $\cL^n$.
We refer to \cite{AFP} for detailed discussions about the
space $BV(\Ome)$ and properties of $BV$ functions.

For any vector $\mathbf{a}\in \mathbf{R}^n$ and any
matrices $\mathcal{A},\, \mathcal{B}\in \mathbf{R}^{n\times m}$ with
$j$th column vectors $\mathcal{A}^{(j)}$ and $\mathcal{B}^{(j)}$, respectively,
we define the following wedge products

\begin{align}\label{e2.3}
&\mathcal{A}\wedge \mathbf{a}:= [\mathcal{A}^{(1)}\wedge \mathbf{a}, 
\cdots, \mathcal{A}^{(m)}\wedge \mathbf{a}],  \qquad
\mathbf{a}\wedge \mathcal{A}:= [\mathbf{a}\wedge\mathcal{A}^{(1)},
\cdots, \mathbf{a}\wedge\mathcal{A}^{(m)} ]  \\
&\mathcal{A}\wedge \mathcal{B}:=\mathcal{A}^{(1)}\wedge \mathcal{B}^{(1)}
+\cdots + \mathcal{A}^{(m)}\wedge \mathcal{B}^{(m)}.
\label{e2.4}
\end{align}
We point out that \reff{e2.4} defines $\mathcal{A}\wedge \mathcal{B}$ to be a
vector instead of a matrix, which seems not to be natural.  However, we 
shall see in Section \ref{sec-6} that it turns out that this is a 
convenient and useful notation.

We conclude this section by citing some known results which will
be used in the later sections. The first result is known as `` the
decisive monotonicity trick", its proof can be found in \cite{Z}.

\begin{lemma}\label{lem2.1}
Let $X$ be a reflexive Banach space and $X'$ denote the dual space of $X$.
Suppose that an operator $\mathcal{F}: X\longrightarrow X'$ satisfies
\begin{itemize}
\item[{\rm (i)}] $\mathcal{F}$ is monotone on $X$, that is, $\langle
\mathcal{F}u-\mathcal{F}v,u-v\rangle
\geq 0$ for all $u,v\in X$;
\item[{\rm (ii)}] $\mathcal{F}$ is hemicontinuous, that is, the function
$t\mapsto \langle \mathcal{F}(u+tv),w\rangle $ is continuous on
$[0,1]$ for all $u, v, w\in X$.
\end{itemize}
In addition, suppose that $u_k\rightarrow u$ and $\mathcal{F}u_k\rightarrow f$
weakly in $X$ and $X'$, respectively, as $n\rightarrow \infty$, and
\[
\overline{\lim_{k\rightarrow \infty}} \langle \mathcal{F} u_k,u_k\rangle \leq
\langle f,u\rangle.
\]
Then
\[
\mathcal{F}u=f.
\]
\end{lemma}

The second lemma is a compactness result, it can be proved following
the proof of Theorem 2.1 of \cite{CHH} (also see Theorem 3 of Chapter 4 of
\cite{Ev}) using the fact that the operator $\Del^\vepsi_p$ is uniformly
elliptic.

\begin{lemma}\label{lem2.2}
For $1\leq p< \infty$, let $p^*=\max\{2,p\}$. For a fixed $\vepsi>0$,
let $\{\bw^{\vepsi,\del}\}_{\del>0}$ be bounded in
$L^\infty((0,T);W^{1,p^*}(\Ome,\bR^n))$ and $\{\bbf^{\vepsi,\del}\}_{\del>0}$
be bounded in $L^1((0,T);L^1(\Ome,$ $\bR^n))$, both uniformly in $\del$.
Moreover, suppose that $\{\frac{\p \bw^{\vepsi,\del}}{\p t}\}_{\del>0}$ 
is bounded in $L^2((0,T);$ $L^2(\Ome,\bR^n))$ uniformly in $\del$ 
and $\bw^{\vepsi,\del}$
satisfies the following equation
\[
\frac{\p \bw^{\vepsi,\del}}{\p t} - \Del^\vepsi_p \bw^{\vepsi,\del}
=\bbf^{\vepsi,\del}\qquad \mbox{in } \Ome_T, \quad \del>0
\]
in the distributional sense. Then $\{\bw^{\vepsi,\del}\}_{\del>0}$
is precompact in $L^q((0,T);W^{1,q}(\Ome,\bR^n))$ for all $1\leq q <p^*$.
\end{lemma}

The last lemma is a variation of Lemma \ref{lem2.2}, its proof can
be found in \cite{Liu1} (also see Lemma 9 of \cite{M1}).

\begin{lemma}\label{lem2.3}
For $1< p< \infty$, let $\{\bw^{\vepsi}\}_{\vepsi>0}$ be bounded in
$L^\infty((0,T);W^{1,p}(\Ome,\bR^n))$ and $\{\bbf^{\vepsi}\}$
be bounded in $L^1((0,T);$ $L^1(\Ome,\bR^n))$, both uniformly in $\vepsi$.
Moreover, suppose that $\{\frac{\p \bw^{\vepsi}}{\p t}\}$ is bounded
uniformly in $\vepsi$ in $L^2((0,T);L^2(\Ome,\bR^n))$ and
$\bw^{\vepsi}$ satisfies the following equation
\[
\frac{\p \bw^{\vepsi}}{\p t} - \Del^\vepsi_p \bw^{\vepsi}
 =\bbf^{\vepsi}\qquad \mbox{in } \Ome_T, \quad \vepsi> 0
\]
in the distributional sense. Then $\{\bw^{\vepsi}\}_{\vepsi>0}$
is precompact in $L^q((0,T);W^{1,q}(\Ome,\bR^n))$ for all $1\leq q <p$.
\end{lemma}

\begin{remark}\label{rem2.1}
In the Lemma \ref{lem2.2}, $\vepsi>0$ is a fixed parameter and the
differential operator $\Del^\vepsi_p$ does not depend on the
variable index $\del$. It can be shown that (cf. Theorem 2.1 of
\cite{CHH}) that the lemma still holds for $\vepsi=0$ when $p\geq
2$. On the other hand, $\vepsi$ is a variable index in Lemma
\ref{lem2.3}, and the operator $\Del^\vepsi_p$ also depends on the
variable index $\vepsi$.
\end{remark}

\section{Well-posedness of the regularized flow \reff{e1.18}-\reff{e1.20}}
\label{sec-3}

In this section we shall analyze the regularized flow
\reff{e1.18}-\reff{e1.20} for each fixed pair of positive numbers
$(\vepsi,\del)$. We establish an energy law, a
maximum principle, uniform (in both $\vepsi$ and $\del$) a priori
estimates, existence and uniqueness of weak and classical
solutions. We begin with a couple of definitions of solutions to
\reff{e1.18}-\reff{e1.20}.

\begin{definition}\label{def3.1}
For $1\leq p < \infty$, a map $\bu^{\vepsi,\del}:
\Ome_T \rightarrow \mathbf{R}^n$ is called a global {\em weak} solution to
\reff{e1.18}-\reff{e1.20} if
\begin{itemize}
\item[{\rm (i)}] $\bu^{\vepsi,\del}\in
L^\infty((0,T);W^{1,p^*}(\Ome,\mathbf{R}^n)) \cap H^1((0,T);
L^2(\Ome, \mathbf{R}^n))$,  for $p^*=\max\{2,p\}$,
\item[{\rm (ii)}] $|\bu^{\vepsi,\del}|\leq 1$ a.e. in $\Ome_T$,
\item[{\rm (iii)}] $\bu^{\vepsi,\del}$ satisfies \reff{e1.18}-\reff{e1.20}
in the distributional sense.
\end{itemize}
\end{definition}

\begin{definition}\label{def3.2}
A weak solution $\bu^{\vepsi,\del}$ to \reff{e1.18}-\reff{e1.20}
is called a {\em strong} solution if~ $\bu^{\vepsi,\del}\in
L^p((0,T);$ $W^{2,p^*}(\Ome,\mathbf{R}^n))
\cap H^1((0,T);L^{p^*}(\Ome,\mathbf{R}^n))$. It
is called a {\em regular} solution if in addition
$\bu^{\vepsi,\del}\in H^1((0,T);$ $W^{1,p^*}(\Ome,\mathbf{R}^n))$.
\end{definition}

\subsection{Energy law and a priori estimates}
Since \reff{e1.18}-\reff{e1.20} is the gradient flow for the
functional $J_{p,\lam}^{\vepsi,\del}$, its regular solutions must
satisfy a dissipative energy law. Indeed, we have the following lemma.

\begin{lemma}\label{lem3.1} Let $\bu_0$ and $\bg$ be  
sufficiently smooth, and suppose that $\bu^{\vepsi,\del}$ is a regular 
solution to \reff{e1.18}-\reff{e1.20}, then $\bu^{\vepsi,\del}$ satisfies the
following energy law
\begin{equation}\label{e3.1}
J_{p,\lam}^{\vepsi,\del}(\bu^{\vepsi,\del}(s)) +\int_0^s
\norm{\bu^{\vepsi,\del}_t(t)}{L^2}^2\, dt
=J_{p,\lam}^{\vepsi,\del}(\bu_0) \qquad\mbox{for a.e. }  s\in [0,T] .
\end{equation}
\end{lemma}

\begin{proof}
Testing equation \reff{e1.18} with $ \bu^{\vepsi,\del}_t$ we get
\begin{align*}
\norm{\bu^{\vepsi,\del}_t(t) }{L^2}^2 &+\frac{d }{d t}\int_\Ome
\Bigl\{\, \frac{b_p(\vepsi)}2 |\nab \bu^{\vepsi,\del}(t)|^2
+\frac{1}{p}|\nab \bu^{\vepsi,\del}(t)|_\vepsi^p \\
&\qquad +\frac{1}{\del} F(\bu^{\vepsi,\del}(t))+\frac{\lam}2
|\bu^{\vepsi,\del}(t)-\bg|^2\,\Bigr\}dx =0.
\end{align*}
The desired identity \reff{e3.1} then follows from
integrating the above equation in $t$ over the interval $[0,s]$
and using the definition of $J_{p,\lam}^{\vepsi,\del}$.
\end{proof}

The above energy law immediately implies the following uniform (in
$\vepsi$ and $\del$) a priori estimates.

\begin{corollary}\label{cor3.1}
Suppose that $\bu_0$ and $\bg$ satisfy 
\begin{equation}\label{IC}
J_{p,\lam}^{\vepsi,\del}(\bu_0)\leq c_0
\end{equation}
for some positive constant $c_0$ independent of $\vepsi$ and $\del$, then 
there exists another positive constant $C:=C(p,\lam,c_0)$ which is also 
independent of $\vepsi$ and $\del$ such that
\begin{eqnarray}\label{e3.2}
&& \norm{\bu^{\vepsi,\del}}{L^\infty(W^{1,p})}
+\norm{\bu^{\vepsi,\del}}{H^1(L^2)} \leq C
\qquad \mbox{for } 1\leq p< \infty, \\
&& \del^{-\frac12} \norm{|\bu^{\vepsi,\del}|^2-1}{L^\infty(L^2)}\leq C
\qquad \mbox{for } 1\leq p< \infty ,  \label{e3.3}\\
&& \norm{|\nab \bu^{\vepsi,\del}|_\vepsi^{p-2}\nab
\bu^{\vepsi,\del}}{L^\infty (L^{p'})} \leq C\qquad \mbox{for }
p'=\frac{p}{p-1},\quad 1\leq p< \infty, \label{e3.4}\\
&& \sqrt{b_p(\vepsi)}\,\norm{\nab \bu^{\vepsi,\del}}{L^\infty(L^2)
}\leq C \qquad \mbox{for } 1\leq p<\infty. \label{e3.5}
\end{eqnarray}
\end{corollary}

Next, using a test function technique of \cite{CHH}, we show that
the modulus $|\bu^{\vepsi,\del}|$ of every weak solution
$\bu^{\vepsi,\del}$ to \reff{e1.18}-\reff{e1.20} satisfies a
maximum principle.

\begin{lemma}\label{lem3.2}
Suppose that $|\bg|\leq 1$ and $|\bu_0|\leq 1$ a.e. in $\Ome$. 
Then weak solutions to \reff{e1.18}-\reff{e1.20} satisfy 
$|\bu^{\vepsi,\del}|\leq 1$ a.e. in $\Ome_T$.
\end{lemma}

\begin{proof}
Define the function
\begin{equation}\label{e3.6}
\chi(z):= \dfrac{(z-1)_+}{z} =\left\{ \begin{array}{ll}
  0 &\quad\mbox{for } 0\leq z\leq 1,\\
  \dfrac{z-1}{z} &\quad\mbox{for } z> 1.
        \end{array} \right.
\end{equation}
It is easy to check that $\chi$ is a nonnegative monotone
increasing function on the interval $[0,\infty)$.

Now testing \reff{e1.18} with $\bv:=\bu^{\vepsi,\del}
\chi(|\bu^{\vepsi,\del}|)$ we get
\begin{align*}
&&\frac12 \frac{d }{d t} \int_{\{ |\bu^{\vepsi,\del}|>1 \} }
\bigl(|\bu^{\vepsi,\del}|-1)^2\, dx
+ \int_{\{ |\bu^{\vepsi,\del}|>1 \} }
\Bigl\{\, b_p(\vepsi) |\nab\bu^{\vepsi,\del}|^2
+|\nab \bu^{\vepsi,\del}|_\vepsi^{p-2} \, |\nab \bu^{\vepsi,\del}|^2\\
&&\quad +\frac{1}{\del} \bigl(\,|\bu^{\vepsi,\del}|^2 -1 \,\bigr)
|\bu^{\vepsi,\del}|^2 
+\lam \bigl(|\bu^{\vepsi,\del}|^2-\bu^{\vepsi,\del}\cdot\bg
\bigr)\,\Bigr\}\, \chi(|\bu^{\vepsi,\del}|)\, dx\\& 
&\quad +\frac14 \int_{\{ |\bu^{\vepsi,\del}|>1 \} } 
\frac{[b_p(\vepsi)+|\nab \bu^{\vepsi,\del}|_\vepsi^{p-2}]\, |\nab |\bu^{\vepsi,\del}|^2 |^2
}{ |\bu^{\vepsi,\del}|^3 }\, dx= 0 .
\end{align*}

Since $\bu^{\vepsi,\del}\cdot\bg 
\leq |\bu^{\vepsi,\del}|\cdot |\bg|
\leq |\bu^{\vepsi,\del}|$,
the second 
integral is nonnegative. The
assertion then follows from integrating the above inequality and
using the assumption $|\bu_0|\leq 1$ a.e. in $\Ome$.
\end{proof}

\subsection{Existence of global weak and classical solutions}\label{sec-3.2}

We now state and prove the existence of global weak and classical
solutions to the regularized flow \reff{e1.18}-\reff{e1.20} for
each fixed pair of positive numbers $(\vepsi,\del)$. Since
$-\Del^\vepsi_p$ is uniformly elliptic for all $1\leq p< \infty$,
the existence of classical solutions follows immediately from
the classical theory of parabolic partial differential
equations (cf. \cite{LSU}).

\begin{theorem}\label{thm3.1}
Let $\Ome\subset \mathbf{R}^m$ be a bounded domain with a smooth
boundary. Suppose that $\bu_0$ and $\bg$ are sufficiently smooth
functions (say, $\bu_0,\bg\in [C^3(\overline{\Ome})]^n$) and 
satisfy \eqref{IC}. 
Then for each fixed pair of positive numbers $(\vepsi, \del)$, 
the regularized flow \reff{e1.18}-\reff{e1.20} possesses a unique global 
classical solution $\bu^{\vepsi,\del}$. Moreover, $\bu^{\vepsi,\del}$
satisfies the following energy law
\begin{equation}\label{e3.8}
J_{p,\lam}^{\vepsi,\del}(\bu^{\vepsi,\del}(s)) +\int_0^s
\norm{\bu^{\vepsi,\del}_t(t)}{L^2}^2\, dt =
J_{p,\lam}^{\vepsi,\del}(\bu_0) \qquad\mbox{for a.e. } s\in [0,T] .
\end{equation}
\end{theorem}

\begin{proof}
The existence and uniqueness follow from an application of the standard results
for parabolic systems, see Chapter 5 of \cite{LSU}. 
\reff{e3.8} follows from Lemma \ref{lem3.1}.
\end{proof}

For less regular functions $\bu_0$ and $\bg$, 
we have the following weaker result.

\begin{theorem}\label{thm3.2}
Let $\Ome\subset \mathbf{R}^m$ be a bounded domain with smooth boundary. 
Suppose that $|\bu_0|\leq 1$ and $|\bg|\leq 1$ a.e. in $\Ome$
and $J_{p,\lam}^{\vepsi,\del}(\bu_0)< \infty$.
Then the regularized flow \reff{e1.18}-\reff{e1.20} has a unique
global weak solution $\bu^{\vepsi,\del}$ in the sense of Definition
\ref{def3.1}.
Moreover, $\bu^{\vepsi,\del}$ satisfies the following energy inequality
\begin{equation}\label{e3.9}
J_{p,\lam}^{\vepsi,\del}(\bu^{\vepsi,\del}(s)) +\int_0^s
\norm{\bu^{\vepsi,\del}_t(t)}{L^2}^2\, dt \leq
J_{p,\lam}^{\vepsi,\del}(\bu_0) \qquad\forall s\in [0,T] .
\end{equation}
\end{theorem}

\begin{proof}
Let $\eta_\rho$ denote any well-known mollifier (cf.
\cite{LSU}), $\bu_{0,\rho}:=\eta_\rho\ast\bu_{0}$
and $\bg_\rho:=\eta_\rho\ast \bg$ denote the mollifications
of $\bu_0$ and $\bg$, respectively. Let $\bu^{\vepsi,\del}_\rho$
denote the classical solution to \reff{e1.18}-\reff{e1.20}
corresponding to the smoothed datum functions $\bu_{0,\rho}$ and
$\bg_\rho$. Hence, $\bu^{\vepsi,\del}_\rho$ satisfies the energy
law \reff{e3.8} with $\bu_{0,\rho}$ and $\bg_\rho$ in 
the place of $\bu_{0}$ and $\bg$.

From Lemma \ref{lem3.2} we know that $|\bu^{\vepsi,\del}_\rho|$
satisfies the maximum principle, thus,
\[
\max_{(x,t)\in \overline{\Ome}_T} |\bu^{\vepsi,\del}_\rho (x,t)|
\leq 1.
\]
Next, since
\[
\lim_{\rho\rightarrow 0} J_{p,\lam}^{\vepsi,\del}(\bu_{0,\rho}^\vepsi)
= J_{p,\lam}^{\vepsi,\del}(\bu_0) < \infty,
\]
the energy law for $\bu^{\vepsi,\del}_\rho$ immediately implies
that $\bu^{\vepsi,\del}_\rho$ satisfies estimates
\reff{e3.2}-\reff{e3.5}, {\em uniformly} in $\rho$ and $\del$.

The remainder of the proof is to extract a convergent subsequence
of $\{\bu^{\vepsi,\del}_\rho\}_{\rho>0}$ and to pass to the limit
as $\rho\rightarrow 0$ in the weak formulation of \reff{e1.18}.
This can be done easily in all terms of \reff{e1.18} except the
nonlinear term in $\Del^\vepsi_p$ (see \reff{e1.21}). To overcome
the difficulty, we appeal to Minty's trick or the ``decisive
monotonicity trick" as described in Lemma \ref{lem2.1}. Since this
part of proof is same as that of Theorem 1.5 of \cite{CHH}, so we
omit it and refer to \cite{CHH} for the details.

The uniqueness of weak solutions follows from a standard
perturbation argument and using the fact that $\Del^\vepsi_p$ is
a monotone operator. Finally, the inequality \reff{e3.9} follows
from setting $\rho\rightarrow 0$ in the energy law \reff{e3.8} for
$\bu^{\vepsi,\del}_\rho$, and using the
lower semicontinuity of $J_{p,\lam}^{\vepsi,\del}$ and
the $L^2$-norm with respect to $L^2$ weak convergence.
\end{proof}

We conclude this section with some remarks.

\begin{remark}\label{rem3.1}
(a). Since $W^{1,1}(\Ome)$ is not a {\em reflexive} Banach space,
Minty's trick would not apply in the case $p=1$ without the
help of the $b_p(\vepsi)\Del$ term in the operator $\Del^\vepsi_p$
(see \reff{e1.21}). On the other hand, in presence of this term,
Minty's trick does apply since we now deal with the Sobolev space
$H^1(\Ome)$ instead of $W^{1,1}(\Ome)$. In fact, if
$b_p(\vepsi)=0$ in \reff{e1.18}, $BV$ solutions are what one can
only expect in general for the regularized flow
\reff{e1.18}-\reff{e1.20} in the case $p=1$ (cf. \cite{FP}).

(b). We also point out that using nonzero parameter $b_p(\vepsi)$
is not necessary in the case $1< p< \infty$. For example, the
conclusion of Theorem \ref{thm3.2} still holds if we replace the
above $b_p(\vepsi)$ and $p^*$ by
\[
\widehat{b}_p(\vepsi)=\left\{ \begin{array}{ll}
                    0 &\quad\mbox{if } 1< p<\infty,\\
                    \vepsi^\alpha &\quad\mbox{if }  p=1 \quad (\alpha>0),
                      \end{array} \right.
\quad\mbox{and}\quad
\widehat{p}^*=\left\{ \begin{array}{ll}
                    p &\quad\mbox{if } 1< p<\infty,\\
                    2 &\quad\mbox{if }  p=1.
                      \end{array} \right.
\]
On the other hand, the conclusion of Theorem \ref{thm3.1} may not be
true any more after this modification.

(c). Theorem \ref{thm3.2} also holds when $\Ome$ is a bounded
Lipschitz domain, one way to prove this assertion is to use the
Galerkin method as done in \cite{CHH}, or to use the finite element
method to be introduced in Section \ref{sec-7}.
\end{remark}

\section{Passing to the limit as $\delta\rightarrow 0$}\label{sec-4}
The goal of this section is to derive the limiting flow of
\reff{e1.18}-\reff{e1.20} as $\del\rightarrow 0$ for each {\em
fixed} $\vepsi>0$. As in \cite{C,CHH,CS,Hu2,M2,S1}, the key ideas
for passing to the limit are to use the compactness result of
Lemma \ref{lem2.2} and to exploit the symmetries of the unit
sphere $S^{n-1}$. Our main result of this section is the following
existence theorem.

\begin{theorem}\label{thm4.1}
Let $p^*:=\max\{2,p\}$ and $\Ome\subset \mathbf{R}^m$ be a bounded
Lipschitz domain. For $1\leq p<\infty$, suppose that $\bu_0\in
W^{1,p^*}(\Ome,\mathbf{R}^n)$, $|\bu_0|= 1$ and $|\bg|\leq 1$
in $\Ome$, then there exists a map $\bu^{\vepsi} \in
L^\infty((0,T);$ $W^{1,p^*}(\Ome,\mathbf{R}^n)) \cap H^1((0,T);
L^2(\Ome, \mathbf{R}^n))$ such that
\begin{equation}\label{e4.1}
|\bu^\vepsi|=1\qquad\mbox{a.e. in  } \Ome_T,
\end{equation}
and $\bu^\vepsi$ is a weak solution (in the distributional sense, see \eqref{e4.25})
to the following problem
\begin{alignat}{2}\label{e4.2}
\bu^{\vepsi}_t - \Del_p^\vepsi \bu^{\vepsi}
+ \lam(\bu^{\vepsi}-\bg)&=\mu_{p,\lam}^\vepsi \bu^\vepsi
&&\qquad\mbox{in }\Ome_T, \\
\cB_p^{\vepsi}\bn &=0 &&\qquad\mbox{on } \p\Ome_T, \label{e4.3}\\
\bu^{\vepsi} &=\bu_0 &&\qquad\mbox{on } \Ome\times \{t=0\}, \label{e4.4}
\end{alignat}
where
\begin{equation}\label{e4.5}
\mu_{p,\lam}^\vepsi:= b_p(\vepsi)|\nab \bu^\vepsi|^2 +
|\nab\bu^\vepsi|_\vepsi^{p-2}\,|\nab\bu^\vepsi|^2 +\lam
(1-\bu^\vepsi\cdot\bg).
\end{equation}
Moreover, $\bu^\vepsi$ satisfies the energy inequality
\begin{equation}\label{e4.6a}
J_{p,\lam}^{\vepsi}(\bu^{\vepsi}(s)) +\int_0^s
\norm{\bu^{\vepsi}_t(t)}{L^2}^2\, dt \leq
J_{p,\lam}^{\vepsi}(\bu_0) \leq J_{p,\lam}^1(\bu_0)
 \qquad\mbox{for a.e. }  s\in [0,T],
\end{equation}
and the additional estimates
\begin{eqnarray}\label{e4.6}
&& \norm{|\nab \bu^{\vepsi}|_\vepsi^{p-2}\nab \bu^{\vepsi}}{L^\infty
(L^{p'})} \leq C\qquad \mbox{for }
p'=\frac{p}{p-1},\quad 1\leq p< \infty, \label{e4.7}\\
&& \norm{\Div\,\cB_p^\vepsi}{L^2(L^1)} \leq C \qquad \mbox{for }
1\leq p<\infty \label{e4.8}, \\
&& \norm{\Div\,(\cB_p^\vepsi\wedge\bu^\vepsi)}{L^2(L^2)} 
\leq C \qquad \mbox{for } 1\leq p<\infty \label{e4.8a},\\
&&\sqrt{b_p(\vepsi)} \norm{\nab \bu^\vepsi}{L^\infty(L^2)} \leq C \label{e4.8b}
\end{eqnarray}
for some positive $\vepsi$-independent constant $C$. Here
$\cB_p^\vepsi$ is the $n\times m$ matrix
\begin{equation}\label{e4.9}
\cB_p^\vepsi:= b_p(\vepsi)\nab \bu^\vepsi + |\nab
\bu^\vepsi|_\vepsi^{p-2} \nab \bu^\vepsi ,
\end{equation}
and
\begin{eqnarray}\label{e4.9a}
J^{\vepsi}_{p,\lam}(\bv)&:=& E^{\vepsi}_p(\bv)
+\frac{\lam}2\int_\Ome |\bv-\bg|^2\, dx .
\end{eqnarray}

\end{theorem}

\begin{proof}
In light of Remark \ref{rem3.1} (c) and the density
argument,  without loss of the generality we assume that $\Ome$ is
a bounded smooth domain, $\bu_0$ and $\bg$ are smooth functions.
On noting that $|\bu_0|=1$ implies that $L^\delta(\bu_0)=0$, then,
$E^{\vepsi,\delta}_p(\bu_0) = E^\vepsi_p(\bu_0)$ in \eqref{e1.14}. 
Since $ E^\vepsi_p(\bu_0) \leq  E^1_p(\bu_0)$, 
Hence, the assumptions on $\bu_0$ and $\bg$ ensure \eqref{IC} holds. 
Since the proof is long, we divide it into three steps.

\medskip
{\em Step 1:  Extracting a convergent subsequence}.
Let $\bu^{\vepsi,\del}$ be the weak solution solution to \reff{e1.18}-\reff{e1.20}
whose existence is established in Theorem \ref{thm3.2}. Since
$E^{\vepsi,\delta}_p(\bu_0)\leq E^1_p(\bu_0)$, then
$J^{\vepsi,\delta}_{p,\lambda}(\bu_0)$ is uniformly bounded with respect 
$\vepsi$ and $\delta$. Hence, \reff{e3.9} implies that
$\bu^{\vepsi,\del}$ satisfies the uniform (in both $\vepsi$ and
$\del$) estimates \reff{e3.2}-\reff{e3.5}, and the maximum principle
$|\bu^{\vepsi,\del}|\leq 1$ on $\overline{\Ome}_T$.

By the weak compactness of $W^{1,p}(\Ome)$ and Sobolev embedding
(cf. \cite{adams75,simon87}), there exists a subsequence of
$\{\bu^{\vepsi,\del}\}_{\del>0}$ (still denoted by the same notation)
and a map $\bu^{\vepsi} \in L^\infty((0,T);
W^{1,p^*}(\Ome,\mathbf{R}^n))\cap H^1((0,T); L^2(\Ome,\mathbf{R}^n))$
such that as $\del\rightarrow 0$
\begin{alignat}{2}\label{e4.10}
\bu^{\vepsi,\del} &\longrightarrow \bu^{\vepsi} &&\qquad\mbox{weakly* in }
L^\infty((0,T);W^{1,p^*}(\Ome,\mathbf{R}^n)), \\
& &&\qquad\mbox{strongly in } L^2((0,T);L^2(\Ome,\mathbf{R}^n)), \label{e4.11}\\
& &&\qquad\mbox{a.e. in } \Ome_T, \label{e4.12} \\
\bu^{\vepsi,\del}_t &\longrightarrow \bu^{\vepsi}_t  &&\qquad\mbox{weakly in }
L^2((0,T);L^2(\Ome,\mathbf{R}^n)), \label{e4.13}\\
|\bu^{\vepsi,\del}| &\longrightarrow 1 &&\qquad\mbox{strongly in }
L^2((0,T);L^2(\Ome)). \label{e4.14}
\end{alignat}

It follows immediately from \reff{e4.12} and \reff{e4.14} that
\begin{equation}\label{e4.15}
|\bu^\vepsi| = 1\qquad\mbox{a.e. in } \Ome_T .
\end{equation}

\medskip
{\em Step 2: Wedge product technique and passing to the limit}.
Since $|\bu^{\vepsi,\del}|\leq 1$ in $\overline{\Ome}_T$, an
application of the Lebesgue dominated convergence theorem yields,
on noting  (\ref{e4.12}),  that
\begin{equation}\label{e4.16}
\bu^{\vepsi,\del} \overset{\del\searrow 0}{\longrightarrow}\bu^{\vepsi}
\qquad\mbox{strongly in } L^r((0,T);L^r(\Ome,\mathbf{R}^n)),\quad\forall
r\in [1, \infty).
\end{equation}

Let
\[
\bbf^{\vepsi,\del}:=\frac{1}{\del}\bigl(1-|\bu^{\vepsi,\del}|^2\bigr) \bu^{\vepsi,\del} 
- \lam\bigl( \bu^{\vepsi,\del}-\bg\bigr),
\]
it follows from the estimate $|\bu^{\vepsi,\del}|\leq 1$ and the inequality 
$|\bu^{\vepsi,\del}|\leq \frac12\bigl( 1+ |\bu^{\vepsi,\del}|^2 \bigr)$ that
\begin{eqnarray*}
\frac{1}{\del}\int_0^T\int_\Ome \bigl( 1-|\bu^{\vepsi,\del}|^2 \bigr)\cdot |\bu^{\vepsi,\del}|
\,dx dt
&&\leq\frac{1}{\del} \int_0^T\int_\Ome \bigl(1-|\bu^{\vepsi,\del}|^2\bigr)^2 \,dx dt \\
&& \quad\qquad
+ \frac{1}{\del} \int_0^T\int_\Ome \bigl(1- |\bu^{\vepsi,\del}|^2\bigr)
\cdot |\bu^{\vepsi,\del}|^2\, dx dt. \nonumber
\end{eqnarray*}
\reff{e3.3} immediately implies that the first term on the right-hand 
side is uniformly bounded in $\delta$. Testing equation \reff{e1.18} 
by $\bu^{\vepsi,\del}$ and using estimates \reff{e3.2}, \reff{e3.4} 
and \reff{e3.5} we conclude that the second term on the 
right-hand side is also uniformly bounded in $\delta$. 
Hence, $\bbf^{\vepsi,\del}$ is uniformly bounded with respect 
to $\delta$ in $L^1((0,T);L^1(\Ome,\mathbf{R}^n))$. 
By Lemma \ref{lem2.2} we have that 
\begin{equation}\label{e4.17}
\nab \bu^{\vepsi,\del} \overset{\del\searrow 0}{\longrightarrow}
\nab\bu^{\vepsi} \qquad\mbox{strongly in }
L^q((0,T);L^q(\Ome,\mathbf{R}^{n\times m})), \quad\forall q\in [1,
p^*).
\end{equation}
This and \reff{e3.4} imply that 
\begin{align}\label{e4.18}
|\nab \bu^{\vepsi,\del}|_\vepsi^{p-2} \nab \bu^{\vepsi,\del}
&\overset{\del\searrow 0}{\longrightarrow}
|\nab\bu^{\vepsi}|_\vepsi^{p-2}\nab\bu^{\vepsi}
\end{align}
weakly* in $L^\infty((0,T);L^{p'}(\Ome,\mathbf{R}^{n\times m}))$ with
$p'=\frac{p}{p-1}$ if $p\neq 1$ and weakly* in $L^{\infty}((0,T);$ 
$L^{\infty} (\Ome,\mathbf{R}^{n\times m}))$ if $p=1$.

Next, taking the wedge product of \reff{e1.18}
with $\bu^{\vepsi,\del}$ yields
\begin{equation}\label{e4.19}
\bu^{\vepsi,\del}_t\wedge \bu^{\vepsi,\del} -\Div\bigl(\,
\cB_p^{\vepsi,\del} \wedge \bu^{\vepsi,\del} \,\bigr) -\lam\,
\bg\wedge \bu^{\vepsi,\del}=0 ,
\end{equation}
where $\cB_p^{\vepsi,\del}$ is defined in \reff{e1.21}.

Testing \reff{e4.19} with any $\bw\in
L^\infty((0,T);W^{1,p^*}(\Ome,\mathbf{R}^n))$ we get
\begin{equation}\label{e4.20}
\int_0^T\Bigl\{ \bigl(\, \bu^{\vepsi,\del}_t\wedge \bu^{\vepsi,\del},\bw\,\bigr)
+\bigl(\,\cB_p^{\vepsi,\del} \wedge \bu^{\vepsi,\del}, \nab
\bw\,\bigr) -\lam\, \bigl(\,\bg\wedge
\bu^{\vepsi,\del},\bw\,\bigr)\Bigr\}\,dt=0 .
\end{equation}
It follows from setting $\del\rightarrow 0$ in \reff{e4.20} and using
\reff{e4.13}, \reff{e4.16} and \reff{e4.18} that
\begin{equation}\label{e4.21}
\int_0^T\Bigl\{ \bigl(\, \bu^{\vepsi}_t\wedge \bu^{\vepsi},\bw\,\bigr) +\bigl(\,
\cB_p^\vepsi \wedge \bu^{\vepsi}, \nab \bw\,\bigr) -\lam\,
\bigl(\,\bg\wedge \bu^{\vepsi},\bw\,\bigr)\Bigr\}\,dt=0 ,
\end{equation}
where $\cB_p^\vepsi$ is given by \reff{e4.9}.

Note that \reff{e4.15} implies
\begin{equation}\label{e4.22}
\bu^{\vepsi}_t \cdot \bu^\vepsi=0, \qquad (\cB_p^\vepsi)^T\, \bu^\vepsi=0
\qquad\mbox{a.e. in } \Ome_T.
\end{equation}
This in turn yields the following identity
\begin{equation}\label{e4.23}
\int_0^T\Bigl\{ \bigl(\, \bu^{\vepsi}_t,\bu^\vepsi\varphi\,\bigr) +\bigl(\,
\cB_p^\vepsi,\nab(\bu^\vepsi \varphi)\,\bigr) +\lam \bigl(\,
\bu^\vepsi-\bg,\bu^\vepsi
\varphi\,\bigr)\Bigr\} dt=\int_0^T\bigl(\,\mu_{p,\lam}^\vepsi,\varphi\,\bigr) dt
\end{equation}
for any $\varphi\in  L^\infty(\Ome_T)\cap L^\infty((0,T);W^{1,p^*}(\Ome))$, where 
$\mu_{p,\lam}^\vepsi$ is defined by \reff{e4.5}.

Finally, for any $\bv\in [C^1({\overline{\Ome}}_T)]^n$, taking $\bw =
\bu^\vepsi\wedge \bv$ in \reff{e4.21},
$\varphi=\bu^\vepsi\cdot\bv$ in \reff{e4.23}, and using the
formula $\bf{a}\cdot (\bf{b}\wedge\bf{c})
=(\bf{a}\wedge\bf{b})\cdot \bf{c}$ yield
\begin{align}\label{e4.24a}
&\int_0^T\Bigl\{\bigl(\bu^{\vepsi}_t,\bu^{\vepsi}\wedge(\bu^\vepsi\wedge\bv)
\bigr) +\bigl(\cB_p^\vepsi,\nab(\bu^{\vepsi}\wedge(\bu^\vepsi\wedge\bv))\bigr)
-\lam\bigl(\bg,\bu^{\vepsi}\wedge(\bu^\vepsi\wedge\bv)\bigr)\Bigr\}\,dt
= 0 ,\\
&\int_0^T\Bigl\{ \bigl(\bu^{\vepsi}_t,\bu^\vepsi(\bu^\vepsi\cdot\bv)\bigr)
+\bigl(\cB_p^\vepsi,\nab (\bu^\vepsi(\bu^\vepsi\cdot\bv))\bigr)
\label{e4.24b} \\
&\hskip 1.7in
+\lam \bigl(\bu^\vepsi-\bg,\bu^\vepsi(\bu^\vepsi\cdot\bv)\bigr)\Bigr\}
\,dt=\int_0^T \bigl(\mu_{p,\lam}^\vepsi\bu^\vepsi,\bv\bigr)\, dt.\nonumber 
\end{align}
Subtracting \reff{e4.24a} from \reff{e4.24b}, and using the identity
\begin{equation*}
\bv=(\bu^\vepsi\cdot\bv)\,\bu^\vepsi - \bu^\vepsi\wedge
(\bu^\vepsi\wedge \bv)
\end{equation*}
we obtain that
\begin{equation}\label{e4.25}
\int_0^T\Bigl\{ \bigl(\, \bu^{\vepsi}_t, \bv\,\bigr) +\bigl(\,\cB_p^\vepsi,
\nab\bv\,\bigr) +\lam \bigl(\, \bu^\vepsi-\bg,
\bv\,\bigr)\Bigr\}\,dt
=\int_0^T \bigl(\,\mu_{p,\lam}^\vepsi \bu^\vepsi,\bv\,\bigr)\,dt
\end{equation}
for any $\bv\in [C^1({\overline{\Ome}}_T)]^n$. This is equivalent to saying that
$\bu^\vepsi$ is a weak solution (in the distributional sense) to \reff{e4.2}-\reff{e4.4}.

\medskip
{\em Step 3: Wrapping up}. We conclude the proof by showing the
estimates \reff{e4.6a}-\reff{e4.8b}. First, \reff{e4.6a} follows
immediately from letting $\del\rightarrow 0$ in \reff{e3.9},
appealing to Fatou's lemma and the lower semicontinuity of
$L^2$- and $L^{p^*}$-norm with respect to $L^2$- and $L^{p^*}$-weak convergence. 
We emphasize again that this is possible in the case $p=1$, because the uniform
(in $\del$) estimate \reff{e3.5} implies that $u^\vepsi \in
L^\infty((0,T);H^1(\Ome,\mathbf{R}^n))$. \reff{e4.7} and \reff{e4.8b} are direct
consequences of \reff{e4.6a}. Finally, the bounds \reff{e4.8} and
\reff{e4.8a} follow immediately from \reff{e4.25} and \reff{e4.21},
respectively. Hence the proof is complete.
\end{proof}

\begin{remark}\label{rem3.2}
If $b_p(\vepsi)=0$ is used in the regularization, the
solutions to \reff{e4.1}-\reff{e4.4} are only expected to belong
to $L^\infty((0,T);[BV(\Ome)]^n)$ in general when $p=1$.
\end{remark}

\section{Passing to the limit as $\vepsi\rightarrow 0$: the case $1<p<\infty$}
\label{sec-5}

In this section, we shall pass to the limit as $\vepsi\rightarrow
0$ in \reff{e4.1}-\reff{e4.4} and show that the limit map is a
weak solution to \reff{e1.8}-\reff{e1.11}.  Since the analysis and
techniques for passing the limit for $1< p<\infty$ and $p=1$ are
quite different, we shall first consider the case $1< p<\infty$ in
this section and leave the case $p=1$ to the next section. We
begin with a definition of weak solutions to
\reff{e1.8}-\reff{e1.11} in the case $1<p<\infty$.

\begin{definition}\label{def5.1}
For $1<p<\infty$, a map $\bu: \Ome_T \rightarrow \mathbf{R}^n$ is called
a global {\em weak} solution to \reff{e1.8}-\reff{e1.11} if
\begin{itemize}
\item[{\rm (i)}] $\bu\in L^\infty((0,T);W^{1,p}(\Ome,\mathbf{R}^n))
\cap H^1((0,T); L^2(\Ome, \mathbf{R}^n))$,
\item[{\rm (ii)}] $|\bu|=1$ a.e. on $\Ome_T$,
\item[{\rm (iii)}] $\bu$ satisfies \reff{e1.8}-\reff{e1.11} in the
distributional sense.
\end{itemize}
\end{definition}

Our main result of this section is the following existence theorem.

\begin{theorem}\label{thm5.1}
Let $1< p<\infty$, suppose that the assumptions on $\bu_0$ and $\bg$ 
in Theorem \ref{thm4.1} still hold.
Then problem \reff{e1.8}-\reff{e1.11} has a weak solution $\bu$ 
in the sense of Definition \ref{def5.1}. Moreover, $\bu$ 
satisfies the energy inequality
\begin{equation}\label{e5.2}
J_{p,\lam}(\bu(s)) +\int_0^s \norm{\bu_t(t)}{L^2}^2\, dt \leq
J_{p,\lam}(\bu_0) \qquad\mbox{for a.e. }  s\in [0,T],
\end{equation}
where $J_{p,\lam}$ is defined by \reff{e1.7}.
\end{theorem}

\begin{proof}
We divide the proof into three steps.

\medskip
{\em Step 1: Extracting a convergent subsequence}.
Let $\bu^\vepsi$ denote the solution of \reff{e4.1}-\reff{e4.4} constructed in
Theorem \ref{thm4.1}.  
From \reff{e4.1}, \reff{e4.6a}--\eqref{e4.8b}, the weak compactness
of $W^{1,p}(\Ome)$ and Sobolev embedding
(cf. \cite{adams75,simon87}), there exists a subsequence of
$\{\bu^{\vepsi}\}_{\vepsi>0}$ (still denoted by the same notation)
and a map $\bu \in L^\infty((0,T); W^{1,p}(\Ome,\mathbf{R}^n))
\cap H^1((0,T); L^2(\Ome,\mathbf{R}^n))$ such that as $\vepsi\rightarrow 0$
\begin{alignat}{2}\label{e5.3}
\bu^{\vepsi} &\longrightarrow \bu &&\qquad\mbox{weakly* in }
L^\infty((0,T);W^{1,p}(\Ome,\mathbf{R}^n)), \\
& &&\qquad\mbox{strongly in } L^2((0,T);L^p(\Ome,\mathbf{R}^n)), \label{e5.4}\\
& &&\qquad\mbox{a.e. in } \Ome_T, \label{e5.5} \\
b_p(\vepsi)\nab\bu^\vepsi &\longrightarrow 0 &&\qquad\mbox{weakly
in }L^2((0,T);L^2(\Ome,\mathbf{R}^{n\times m})), \label{e5.6a}\\
\bu^{\vepsi}_t &\longrightarrow \bu_t  &&\qquad\mbox{weakly in }
L^2((0,T);L^2(\Ome,\mathbf{R}^{n})). \label{e5.6}
\end{alignat}

It follows immediately from \reff{e5.5} and \reff{e4.15} that
\begin{equation}\label{e5.8}
|\bu| = 1\qquad\mbox{a.e. in } \Ome_T .
\end{equation}

\medskip
{\em Step 2: Passing to the limit.} First, by \reff{e5.5} 
and the Lebesgue dominated convergence theorem we have that
\begin{equation}\label{e5.9}
\bu^{\vepsi} \overset{\vepsi\searrow 0}{\longrightarrow}\bu
\qquad\mbox{strongly in } L^r((0,T);L^r(\Ome,\mathbf{R}^n)),\quad\forall
r\in [1, \infty).
\end{equation}

Next, let $\bbf^\vepsi:=\mu_{p,\lam}^\vepsi \bu^\vepsi-\lam\,(\bu^\vepsi-\bg)$.
Clearly, $\bbf^\vepsi\in L^1((0,T);L^1(\Ome;\mathbf{R}^n))$ and is uniformly
bounded, on noting \reff{e4.5}-\reff{e4.6a}. By Lemma \ref{lem2.3} we get
\begin{equation}\label{e5.10}
\nab \bu^{\vepsi} \overset{\vepsi\searrow 0}{\longrightarrow}
\nab\bu\qquad\mbox{strongly in }
L^q((0,T);L^q(\Ome,\mathbf{R}^{n\times m})), \quad\forall q\in [1, p),
\end{equation}
which, \reff{e4.7} and \reff{e5.6a} imply that
\begin{align}\label{e5.11}
&|\nab \bu^\vepsi|_\vepsi^{p-2} \nab \bu^\vepsi
\overset{\vepsi\searrow 0}{\longrightarrow} |\nab\bu|^{p-2}\nab\bu
\quad\mbox{weakly* in } L^\infty((0,T);L^{p'}(\Ome,\mathbf{R}^{n\times m})), \\
&\cB_p^\vepsi \overset{\vepsi\searrow 0}{\longrightarrow}
\cB_p:=|\nab\bu|^{p-2}\nab\bu \quad\mbox{weakly in }
L^2((0,T);L^{p^\prime_*}(\Ome,\mathbf{R}^{n\times m})) , \label{e5.12}
\end{align}
where $p'=\frac{p}{p-1}$ and $p^\prime_*:=\min\{2, p^\prime\}$.

It then follows from taking $\vepsi\rightarrow 0$ in \reff{e4.21}
and using \reff{e5.6}, \reff{e5.9}, and \reff{e5.12} that
\begin{equation}\label{e5.13}
\int_0^T\Bigl\{ \bigl(\, \bu_t\wedge \bu,\bw\,\bigr) +\bigl(\,
\cB_p \wedge \bu, \nab \bw\,\bigr) -\lam\,
\bigl(\,\bg\wedge \bu,\bw\,\bigr) \Bigr\}\,dt =0 
\end{equation}
for any $\bw\in C^1(\overline{\Ome}_T)$. Since 
$\cB_p\in L^\infty((0,T);L^{p'}(\Ome,\mathbf{R}^{n\times m}))$, 
by the standard density argument one can show that
\eqref{e5.13} also holds for all $\bw\in L^\infty((0,T);W^{1,p}(\Ome,\mathbf{R}^n))
\cap L^\infty(\Ome_T)$.

Since $|\bu|=1$ a.e. in $\Ome_T$, there holds the following identity,
which is analogous to \reff{e4.23},
\begin{equation}\label{e5.14}
\int_0^T\Bigl\{ \bigl(\, \bu_t,\bu\varphi\,\bigr) +\bigl(\,\cB_p,\nab(\bu
\varphi)\,\bigr) +\lam \bigl(\, \bu-\bg ,\bu
\varphi\,\bigr) \Bigr\}\,dt
=\int_0^T \bigl(\,\mu_{p,\lam},\varphi\,\bigr)\,dt
\end{equation}
for any $\varphi\in L^\infty(\Ome_T)\cap L^\infty((0,T);W^{1,p}(\Ome,\mathbf{R}^n))$; 
where $\mu_{p,\lam}$ is defined by \reff{e1.11a}.

Finally, for any $\bv\in [C^1({\overline{\Ome}}_T)]^n$, on choosing $\bw =
\bu\wedge \bv$ in \reff{e5.13} and $\varphi=\bu\cdot\bv$ in
\reff{e5.14}, subtracting the resulting equations and using the
identity
\begin{equation*}
\bv=(\bu\cdot\bv)\,\bu - \bu\wedge (\bu\wedge \bv);
\end{equation*}
we obtain that
\begin{equation}\label{e5.15}
\int_0^T\Bigl\{ \bigl(\, \bu_t, \bv\,\bigr) +\bigl(\,\cB_p,
\nab \bv\,\bigr) +\lam \bigl(\, \bu-\bg , \bv\,\bigr)\Bigr\}\,dt
= \int_0^T \bigl(\,\mu_{p,\lam} \bu,\bv\,\bigr)\,dt
\end{equation}
for any $\bv\in [C^1({\overline{\Ome}}_T)]^n$.
Hence, $\bu$ is a weak solution to \reff{e1.8}-\reff{e1.11}.

\medskip
{\em Step 3: Wrapping up}. We conclude the proof by showing the
energy inequality \reff{e5.2}. First, notice that \reff{e4.6a} implies
that
\begin{equation}\label{e5.16}
\int_\Ome\, \Bigl\{ \frac{1}{p} |\nab \bu^\vepsi(s)|^p
+ \frac{\lam}{2} |\bu^\vepsi(s)-\bg|^2 \,\Bigr\} dx
+ \int_0^s \norm{\bu^{\vepsi}_t(t)}{L^2}^2\, dt
\leq J_{p,\lam}^{\vepsi}(\bu_0) \quad\forall s\in [0,T].
\end{equation}
Then \reff{e5.2} follows from taking $\vepsi\rightarrow 0$ in
\reff{e5.16}, using Fatou's lemma and the lower semicontinuity
of the $L^p$-norm with respect to $L^p$-weak convergence.
The proof is complete.
\end{proof}

\begin{remark}\label{rem5.2}
(a). We remark that it was proved in \cite{Co,Hu1,M2} that weak
solutions to \reff{e1.8}-\reff{e1.11} are not unique in general.

(b). Although the above proof is carried out for any $\alpha> 0$ in
the definition of $b_p(\vepsi)$ (cf. \reff{e1.13}), $\alpha$
should be chosen large enough so that the error due to the
perturbation term $b_p(\vepsi)\Del$ is much smaller than the error
due to other regularization terms in numerical simulations. 

(c). The existence result of Theorem \ref{thm5.1} is established 
under the assumption $\bu_0\in W^{1,p*}(\Ome,\mathbf{R}^n)$
with $p*=\max\{p,2\}$. This condition can be weakened to 
$\bu_0\in W^{1,p}(\Ome,\mathbf{R}^n)$ in the case $1<p<2$ 
by a smoothing technique.
\end{remark}

\section{Passing to the limit as $\vepsi\rightarrow 0$: the case $p=1$}
\label{sec-6}

In this section, we consider the case $p=1$ and establish 
the existence of global weak solutions for the
$1$-harmonic map heat flow \reff{e1.8}-\reff{e1.11} by passing to
the limit as $\vepsi\rightarrow 0$ in \reff{e4.2}-\reff{e4.4}.
There are two main difficulties which prevent one to repeat the
analysis and techniques of the previous section. Firstly, the
compactness result of Lemma \ref{lem2.3} does not hold any more
when $p=1$. Secondly, since the sequence
$\{\bu^\vepsi\}_{\vepsi>0}$ is uniformly bounded only in
$L^\infty((0,T);[W^{1,1}(\Ome)\cap L^\infty(\Ome)]^n)$, and
$W^{1,1}(\Ome,\mathbf{R}^n)$ is not a reflexive Banach space,
hence, the limiting map $\bu$ now belongs to $L^\infty((0,T);
[BV(\Ome)\cap L^\infty(\Ome)]^n)$, i.e., $\bu(t)$ is only a map of
bounded variation. As expected, these two difficulties make the
passage to the limit as $\vepsi\rightarrow 0$ become more difficult and
delicate. 

\subsection{Technical tools and lemmas} \label{sec-6.1}

In this subsection, we shall cite some technical tools and 
lemmas in order to develop a weak solution concept to be given in the
next subsection for the $1$-harmonic map heat flow. 
Specially, we need the pairings $\mathcal{A}\cdot D\mathbf{v}$
and $\mathcal{A}\wedge D\mathbf{v}$  between a tensor field
$\mathcal{A}$ and a $BV$-vector field $\mathbf{v}$, which was
developed in \cite{F1} as a generalization of the pairing
$\mathbf{b}\cdot Dv$ between a vector field  $\mathbf{b}$
and a $BV$-function $v$ developed in \cite{An,CF}.

We recall from \cite{F1} the space of {\em divergence-$L^q$ tensors}
\begin{equation}\label{e6.12}
\bY(\Ome)_q:=\{ \cA\in L^\infty(\Ome;\bR^{n\times m});\, 
\Div \cA\in L^q(\Ome;\bR^n) \} \qquad\mbox{for } 1\leq q <\infty,
\end{equation}
and that $\cA\cdot D\bv$ and $\cA\wedge D\bv$ are defined as
\begin{definition}\label{def6.0}
For any $\cA\in \bY(\Ome)_1$ and $\bv \in [BV(\Ome)\cap L^\infty(\Ome)]^n$,
we define $\cA\cdot D\bv$ and $\cA\wedge D\bv$ to be the functionals on 
$C^\infty_0(\Ome)$ and $[C^\infty_0(\Ome)]^n$, respectively by 
\begin{align}\label{e6.13}
&\langle \cA\cdot D\bv,\psi\rangle
:=-\int_\Ome\, (\cA^T\bv)\cdot \nab\psi\, dx
-\int_\Ome(\Div\cA\cdot\bv) \psi\, dx \qquad\forall \psi\in C^\infty_0(\Ome),
\\
&\langle\cA\wedge D\bv, \bw\rangle 
:=-\int_\Ome (\cA\wedge \bv)\cdot \nab\bw\, dx
-\int_\Ome (\Div\cA\wedge \bv)\cdot\bw \,dx \,\,
\forall\bw\in [C^\infty_0(\Ome)]^n;  \label{e6.14}
\end{align}
where $\cA^T$ stands for the matrix transpose of $\cA$ and the notation
\eqref{e2.4} is used in \eqref{e6.14}.
\end{definition}

We now list some properties of the pairings $\cA\cdot D\bv$ and $\cA\wedge D\bv$,
and refer to Section 2 of \cite{F1} for their proofs. The first lemma 
declares that $\cA\cdot D\bv$ and $\cA\wedge D\bv$ are Radon measures in $\Ome$.

\begin{lemma}\label{lem6.2} 
For any Borel set $E\subset\Ome$, there hold
\begin{align*}
&\bigl|\langle\cA\cdot D\bv,\psi\rangle\bigl|
\leq \max_{x\in E}|\psi(x)|\cdot \norm{\cA}{L^\infty(E,\bR^{n\times m})}
\cdot |D\bv|(E)\quad \forall \psi\in C_0(E),\\
&\bigl|\langle\cA\wedge D\bv,\bw\rangle\bigl|
\leq \max_{x\in E}|\bw(x)|\cdot \norm{\cA}{L^\infty(E,\bR^{n\times m})}
\cdot |D\bv|(E)\quad\forall \bw\in [C_0(E)]^n. 
\end{align*}
Hence, it follows from the Riesz Theorem (cf. Theorem 1.54 of \cite{AFP}) 
that both functionals $\cA\cdot D\bv$ and $\cA\wedge D\bv$ are Radon 
measures in $\Ome$.
\end{lemma}

\begin{corollary}\label{cor6.1}
The measures $\cA\cdot D\bv$, $|\cA\cdot D\bv|$, $\cA\wedge D\bv$, and
$|\cA\wedge D\bv|$ all are absolutely continuous with respect to the measure
$|D\bv|$ in $\Ome$. Moreover, there hold inequalities
\begin{eqnarray}\label{e6.18}
\bigl|(\cA\cdot D\bv)(E)\bigr| &\leq& \bigl|\cA\cdot D\bv\bigr|(E)
\leq \norm{\cA}{L^\infty(E',\bR^{n\times m})}\cdot |D\bv|(E), \\
\bigl|(\cA\wedge D\bv)(E)\bigr| &\leq& \bigl|\cA\wedge D\bv\bigr|(E)
\leq \norm{\cA}{L^\infty(E',\bR^{n\times m})}\cdot |D \bv|(E) \label{e6.19}
\end{eqnarray}
for all Borel sets $E$ and for all open sets $E'$ such 
that $E\subset E'\subset\Ome$.

Hence, by the Radon-Nikod\'ym Theorem (cf. Theorem 1.28 of \cite{AFP}), there
exist $|D\bv|$-measurable functions $\Theta:=\Theta(\cA,D\bv,x): 
\Ome\rightarrow \bR$,
and $\Lambda:=\Lambda(\cA,D\bv,x): \Ome\rightarrow \bR^n$ such that
\begin{alignat}{2}\label{e6.20}
(\cA\cdot D\bv)(E) &= \int_E\, \Theta\, d|D\bv|, &&\quad\mbox{and}\quad 
\norm{\Theta}{L^\infty(\Ome,|D\bv|)} \leq \norm{\cA}{L^\infty(\Ome,\bR^{n\times m})};\\
(\cA\wedge D\bv)(E) &= \int_E\, \Lambda\, d|D\bv|, &&\quad\mbox{and}\quad
\norm{\Lambda}{L^\infty(\Ome,|D\bv|)} \leq \norm{\cA}{L^\infty(\Ome,\bR^{n\times m})}
\label{e6.21}
\end{alignat}
for all Borel sets $E\subset \Ome$.
\end{corollary}

The second lemma declares that every $\cA\in \bY(\Ome)_1$ has a well-behaved 

{\em traction} $\cA\bn$ on the boundary of a Lipschitz domain $\Ome$. 

\begin{lemma}\label{lem6.3} 
Let $\Ome$ be a bounded domain with a Lipschitz continuous boundary 
$\p\Ome$ in $\bR^m$, then there exists a linear operator 
$\beta:\bY(\Ome)_1\rightarrow L^\infty(\p\Ome;\bR^n)$ such that
\begin{align}\label{e6.22}
\norm{\beta(\cA)}{L^\infty(\p\Ome,\bR^n)} 
&\leq \norm{\cA}{L^\infty(\Ome,\bR^{n\times m})}, \\
\langle\cA,\bv\rangle_{\p\Ome} &=\int_{\p\Ome}\, \beta(\cA)(x)\,\bv(x)\,d\cH^{m-1}
\quad\forall \bv\in [BV(\Ome)\cap L^\infty(\Ome)]^n, \label{e6.23} \\
\beta(\cA)(x)&= \cA(x)\bn(x) \qquad\forall x\in\p\Ome,\;
\cA\in C^1(\overline{\Ome},\bR^{n\times m}). \label{e6.24}
\end{align}
\end{lemma}

\begin{remark}
Since $\beta(\cA)$ is a weakly defined traction of $\cA$ on $\p\Ome$, 
hence we shall use $\cA\bn$ to denote $\beta(\cA)$ in the rest of this section.
\end{remark}

The third lemma declares that the following hold. 

\begin{lemma}\label{lem6.4}
Let $\Ome$ be a bounded domain with a Lipschitz continuous boundary 
$\p\Ome$ in $\bR^m$, then for any $\cA\in \bY(\Ome)_1$ and 
$\bv\in [BV(\Ome)\cap L^\infty(\Ome)]^n$ there hold identities
\begin{eqnarray}\label{e6.25}
\int_\Ome\,\Div\cA\cdot \bv\, dx + (\cA\cdot D\bv)(\Ome) 
&=&\int_{\p\Ome}\, \cA\bn\cdot\bv\, d\cH^{m-1},\\
\int_\Ome\,\Div\cA\wedge \bv\, dx + (\cA\wedge D\bv)(\Ome) 
&=&\int_{\p\Ome}\, \cA\bn\wedge\bv\, d\cH^{m-1}. \label{e6.26}
\end{eqnarray}
\end{lemma}

The fourth and fifth lemmas states continuity results for the measure 
$\cA\cdot D\bv$ and $\cA\wedge D\bv$ with respect to $\cA$ and
$\bv$, respectively.

\begin{lemma}\label{lem6.8}
Let $\cA_j,\cA \in \bY(\Ome)_1$ and suppose that
\begin{eqnarray*}
&&\cA_j \longrightarrow \cA \quad\mbox{weakly$*$ in } L^\infty(E),\\
&&\Div\cA_j \longrightarrow \Div\cA \quad\mbox{weakly in } L^1(E)
\end{eqnarray*}
for all open sets $E\subset\subset \Ome$.  Then for all $\bv\in 
[BV(\Ome)\cap L^\infty(\Ome)]^n$ the following  hold
\begin{eqnarray}\label{e6.26a}
&&\cA_j\cdot D\bv \longrightarrow 
\cA\cdot D\bv \quad\mbox{weakly$*$ in } \cM(\Ome),\\
&&\cA_j\wedge D\bv \longrightarrow 
\cA\wedge D\bv \quad\mbox{weakly$*$ in } [\cM(\Ome)]^n, \label{e6.26b}
\end{eqnarray}
and 
\begin{eqnarray}\label{e6.26c}
&&\Theta(\cA_j,D\bv,\cdot) \longrightarrow \Theta(\cA,D\bv,\cdot)
\quad\mbox{weakly$*$ in } L^\infty(E) \,\mbox{for all } E\subset\subset\Ome,\\
&& \Lambda(\cA_j,D\bv,\cdot) \longrightarrow \Lambda(\cA,D\bv,\cdot)
\quad\mbox{weakly$*$ in } [L^\infty(E)]^n \,\mbox{for all } E\subset\subset\Ome.
\label{e6.26d}
\end{eqnarray}
Here ``$E\subset\subset \Ome$" means that $E$ is compactly contained in $\Ome$;
that is,  $E\subset \overline{E} \subset \Ome$ and $\overline{E}$ is compact.
\end{lemma}

\begin{lemma}\label{lem6.9}
Let $\cA\in \bY(\Ome)_1$ and $\bv\in [BV(\Ome)\cap L^\infty(\Ome)]^n$.
Suppose that $\{\bv_j\} \subset [C^\infty(\Ome)\cap BV(\Ome)]^n$
strictly converges to $\bv$ (cf. Definition 3.14 of \cite{AFP}). Then
\begin{eqnarray}\label{e2.36h}
&&\cA\cdot D\bv_j \longrightarrow \cA\cdot D\bv \quad\mbox{weakly$*$ in }
\cM(\Ome),\\
&&\cA\wedge D\bv_j \longrightarrow \cA\wedge D\bv \quad\mbox{weakly$*$ in }
[\cM(\Ome)]^n. \label{e2.36i}
\end{eqnarray}
Moreover,
\begin{eqnarray}\label{e2.36j}
&&\int_\Ome \cA\cdot D\bv_j\, dx\longrightarrow
\int_\Ome \cA\cdot D\bv ,\\
&&\int_\Ome \cA\wedge D\bv_j\, dx\longrightarrow
\int_\Ome \cA\wedge D\bv. \label{e2.36k}
\end{eqnarray}
\end{lemma}

The sixth lemma gives the precise representations for the 
density functions $\Theta$ and $\Lambda$ defined in 
Corollary \ref{cor6.1}. 

\begin{lemma}\label{lem6.5}
(i) If $\cA\in \bY(\Ome)_1\cap C(\Ome, \bR^{n\times m})$ and 
$\bv\in [BV(\Ome)\cap L^\infty(\Ome)]^n$, then there hold
\begin{equation}\label{e6.27}
\Theta(\cA,D\bv,x)=\cA(x)\cdot \frac{D\bv}{|D\bv|}(x),\qquad
\Lambda(\cA,D\bv,x)=\cA(x)\wedge \frac{D\bv}{|D\bv|}(x)
\end{equation}
$|D\bv|-\mbox{a.e. in } \Ome$

(ii) If $\cA\in \bY(\Ome)_1$ and $\bv\in [BV(\Ome)\cap L^\infty(\Ome)]^n$, 
then there hold 
\begin{equation}\label{e6.28} 
\Theta(\cA,D\bv,x)=\cA(x)\cdot \frac{D\bv}{|D\bv|}(x),\qquad \Lambda(\cA,D\bv,x)=\cA(x)\wedge \frac{D\bv}{|D\bv|}(x)
\end{equation}
$|D\bv|^a-\mbox{a.e. in } \Ome$. Where $\frac{D\bv}{|D\bv|}$ denotes 
the density function of the measure $D\bv$ with respect to the measure $|D\bv|$,
and $|D\bv|^a$ denotes the absolute continuous part of the measure $|D\bv|$ with 
respect to the Lebesgue measure $\cL^n$.
\end{lemma}

\medskip
Next, we recall from \cite{F1} the space of {\em divergence-measure tensors}
\begin{equation}\label{e6.31}
\mathcal{DT}(\Ome):=\bigl\{\, \cA\in L^\infty(\Ome,\mathbf{R}^{n\times m});
\, \Div\cA\in [\cM(\Ome)]^n \,\bigr\},
\end{equation}
and briefly discuss two of its important properties.
Firstly, like in the case of the space of the divergence-$L^1$ tensors 
$\bY(\Ome)_1$,  Definition \ref{def6.0} is still valid
for $\cA\in \mathcal{DT}(\Ome)$ and $\bv\in [BV(\Ome)\cap L^\infty(\Ome)
\cap C(\Ome)]^n$ (cf. \cite{An}). Secondly, there is a well-behaved traction 
$\cA\bn$ for every $\cA\in \mathcal{DT}(\Ome)$.  

\begin{lemma}\label{lem6.6} 
Let $\Ome$ be a bounded domain with Lipschitz continuous boundary 
$\p\Ome$ in $\bR^m$, then there exists a linear operator $\alpha:\mathcal{DT}(\Ome)
\rightarrow L^\infty(\p\Ome;\bR^n)$ such that
\begin{align}\label{e6.32}
&\norm{\alpha(\cA)}{L^\infty(\p\Ome,\bR^n)} 
\leq \norm{\cA}{L^\infty(\Ome,\bR^{n\times m})}, \\
&\langle\cA,\bv\rangle_{\p\Ome} =\int_{\p\Ome}\, \alpha(\cA)(x)\,\bv(x)\,d\cH^{m-1}
\qquad\forall \bv\in [BV(\Ome)\cap L^\infty(\Ome)\cap C(\Ome)]^n, 
\label{e6.33} \\
&\alpha(\cA)(x)= \cA(x)\bn(x) \qquad\forall x\in\p\Ome,\, 
\cA\in C^1(\overline{\Ome},\bR^{n\times m}). \label{e6.34}
\end{align}
Moreover, for any $\cA\in \mathcal{DT}(\Ome)$ and 
$\bv\in [BV(\Ome)\cap L^\infty(\Ome)\cap C(\Ome)]^n$, 
let $\cA\bn:=\alpha(\cA)$ on $\p\Ome$, then there hold the following 
Green's formulas
\begin{eqnarray}\label{e6.35}
(\Div\cA\cdot\bv)(\Ome) + (\cA\cdot D\bv)(\Ome) 
&=&\int_{\p\Ome}\, \cA\bn\cdot\bv\, d\cH^{m-1},\\
(\Div\cA\wedge \bv)(\Ome) + (\cA\wedge D\bv)(\Ome) 
&=&\int_{\p\Ome}\, \cA\bn\wedge\bv\, d\cH^{m-1}.  \label{e6.36}
\end{eqnarray}
\end{lemma}

Thirdly, there holds the following product rule .

\begin{lemma}\label{lem6.7}
For any $\cA\in \mathcal{DT}(\Ome)$ and $\bv\in [BV(\Ome)\cap L^\infty(\Ome)]^n$,
the identities
\begin{eqnarray}\label{e6.37}
\Div(\cA^T\bv)&=&(\Div\cA)\cdot \overline{\bv} + \overline{\cA\cdot D\bv}, \\
\Div(\cA\wedge\bv)&=&(\Div\cA)\wedge\overline{\bv} + \overline{\cA\wedge D\bv}.
\label{e6.38} 
\end{eqnarray}
hold in the sense of Radon measures in $\Ome$. Where $\overline{\bv}$ denotes
the limit of a mollified sequence for $\bv$ through a positive symmetric mollifier,
$\overline{\cA\cdot D\bv}$ (resp. $\overline{\cA\wedge D\bv}$) is a Radon measure 
which is absolutely continuous with respect to the measure $|D\bv|$, and whose 
absolutely continuous part $(\overline{\cA\cdot D\bv})^a$ (resp. 
$(\overline{\cA\wedge D\bv})^a$) with respect to the Lebesgue measure $\cL^m$ 
in $\Ome$ coincides with $\cA\cdot(\nab\bv)^a$ (resp. $\cA\wedge(\nab\bv)^a$)
almost everywhere in $\Ome$, that is, $(\overline{\cA\cdot D\bv})^a
=\cA\cdot (\nab\bv)^a$ (resp. $(\overline{\cA\wedge D\bv})^a
=\cA\wedge (\nab\bv)^a$), $\cL^m-$ a.e. in $\Ome$.
\end{lemma}

\begin{remark}
\reff{e6.37} and \reff{e6.38} hold without all the overbars if either
$\bv\in [BV(\Ome)\cap L^\infty(\Ome)\cap C(\Ome)]^n$ or $\cA\in \bY(\Ome)_1$.
\end{remark}

\begin{remark}
It should be note that the results of Lemmas \ref{lem6.8}-\ref{lem6.5} also hold
for the tensor fields in $\mathcal{DT}(\Ome)$ and
the vector fields in $[BV(\Ome)\cap L^\infty(\Ome)\cap C(\Ome)]^n$.
\end{remark}

\subsection{Existence of weak solutions of $1$-harmonic map heat flow} \label{sec-6.2}

Throughout the rest of this paper we let $B_1(\bR^{n\times m})$ denote the
unit ball in the Euclidean space $\bR^{n\times m}$;  that is,
\[
B_1(\bR^{n\times m})=\Bigl\{ \cA\in \mathbf{R}^{n\times m};\, 
|\cA|:=\Bigl(\sum_{j=1}^m \sum_{i=1}^n \cA_{ij}^2\Bigr)^{\frac12} \leq 1 \Bigr\}.
\]
In addition, let
\[
\mathbf{\sigma}_{\lam}^\vepsi:=\mu_{1,\lam}^\vepsi\bu^\vepsi,\qquad
\cB^\vepsi :=\frac{\nab\bu^\vepsi}{|\nab\bu^\vepsi|_\vepsi}.
\]
Hence $\cB_1^\vepsi=b_1(\vepsi)\nab\bu^\vepsi + \cB^\vepsi$ (cf. \reff{e4.9}).

We now give a definition of weak solutions to
\reff{e1.8}-\reff{e1.11} in the case $p=1$.

\begin{definition}\label{def6.1}
For $p=1$, a map $\bu: \Ome_T \rightarrow \mathbf{R}^n$ is called
a global {\em weak} solution to \reff{e1.8}-\reff{e1.11} if there exists
a tensor (or matrix-valued function) $\cB$ such that
\begin{itemize}
\item[{\rm (i)}] $\bu\in L^\infty((0,T);[BV(\Ome)\cap L^\infty(\Ome)]^n)
     \cap H^1((0,T); L^2(\Ome, \mathbf{R}^n))$,
\item[{\rm (ii)}] $|\bu|=1\quad \cL^{m+1}-$a.e. in $\Ome_T$,
\item[{\rm (iii)}] $\cB\in L^\infty((0,T);
L^\infty(\Ome,B_1(\mathbf{R}^{n\times m})))\cap L^2((0,T);\mathcal{DT}(\Ome))$,
\item[{\rm (iv)}] $\bu$ and $\cB$ satisfy $\cB\wedge\bu\in L^2((0,T);\bY(\Ome)_2)$, 
 $\cB^T \bu=0$, and $\cB\cdot (D\bu)^a=|D \bu|^a$ $\cL^{m+1}-$a.e. in $\Ome_T$. 
\item[{\rm (v)}] $\cB\bn=0$ on $\p\Ome_T$ in the sense of Lemma \ref{lem6.6}.
\item[{\rm (vi)}] there holds the identity
\begin{eqnarray*}
\int_0^T\int_\Ome\, \Bigl\{ \bu_t\cdot \bv +\cB\cdot \nab\bv 
+\lam (\bu-\bg)\cdot \bv \Bigr\}\, dxdt = 
\int_0^T \Bigl(\int_\Ome \bv\, d\mathbf{\sigma}_{\lam}\Bigr)\, dt
\end{eqnarray*}
for any $\bv\in C^1(\overline{\Ome}_T)$. Where $\mathbf{\sigma}_{\lam}$ 
denotes the vector-valued Radon measure
\begin{eqnarray}\label{e6.39}
\mathbf{\sigma}_{\lam} 
=(\cB\wedge\bu)\wedge D\bu + \lam\,(1-\bg\cdot\bu)\,\bu .
\end{eqnarray}
Moreover, the Radon measure $(\cB\wedge\bu)\wedge D\bu$ is absolutely 
continuous with respect to the measure $|D\bu|$, and  for
$\cL^1-$ a.e. $t\in (0,T)$ there exists a function $\Phi(\cB,D\bu,x,t)
:\Ome\rightarrow \bR$ such that
\end{itemize}
\begin{eqnarray*}\label{e6.40}
&&((\cB\wedge\bu)\wedge D\bu) (E) = \int_E\, \Phi(\cB\wedge\bu,D\bu,x,t)\, d|D\bu|
\qquad\mbox{for all Borel sets } E\subset \Ome, \\
&&\norm{\Phi(t)}{L^\infty(\Ome,|D\bu|)}
\leq \norm{\cB(t)}{L^\infty(\Ome,\bR^{n\times m})} 
\quad\mbox{for $\cL^1$ a.e. } t\in (0,T), \label{e6.41} \\
&&\Phi(\cB\wedge\bu,D\bu,x,t) = (\cB(t)\wedge \bu)\wedge \frac{D\bu}{|D\bu|}
\quad |D\bu|^a \mbox{a.e. in }\Ome, \mbox{ $\cL^1$ a.e. } t\in (0,T),
\label{e6.42}\\
&& ((\cB\wedge\bu)\wedge D\bu)^a= (\cB\wedge\bu)\wedge (D\bu)^a
= |D\bu|^a\bu\qquad \cL^{m+1}\mbox{ a.e. in $\Ome_T$.} \label{e6.43}
\end{eqnarray*}

\end{definition}

\begin{remark}\label{rem6.1}
(a). The tensor field $\cB$ extends $\frac{D \bu}{|D\bu|}$ as
a calibration across the gulfs.

(b). If $\bu(t)\in W^{1,1}_{\tiny{\rm loc}}(\Ome)$ and $\cB(t)\in \bY(\Ome)_1$,
using (ii) and (iv), and the identity $(\mathbf{a}\wedge\mathbf{b})\wedge \mathbf{c}
=(\mathbf{a}\cdot \mathbf{c}) \mathbf{b}-(\mathbf{b}\cdot \mathbf{c}) \mathbf{a}$ we get
that $(\cB\wedge\bu)\wedge D\bu =|\nab\bu|\,\bu$. Hence, \reff{e6.39} can be rewritten as
\begin{equation}\label{e6.44}
\mathbf{\sigma}_{\lam} 
=|\nab\bu|\,\bu +\lam\,(1-\bg\cdot\bu)\,\bu.
\end{equation}
Thus, \reff{e6.39} is a weak form of \reff{e6.44} since $|D\bu|\,\bu$ may not
be defined for $\bu(t)\in [BV(\Ome)\cap L^\infty(\Ome)]^n$ and
$\cB(t) \in \mathcal{DT}(\Ome)$.
\end{remark}

Our main result of this section is the following existence theorem.

\begin{theorem}\label{thm6.1}
Let $p=1$, suppose that $\bu_0\in H^1(\Ome,\mathbf{R}^n)$ 
and $\bg\in L^2(\Ome,\mathbf{R}^n)$ with $|\bu_0|= 1$ and 
$|\bg|\leq 1$ a.e. in $\Ome$. Then, problem \reff{e1.8}-\reff{e1.11}
has a global weak solution $\bu$ in the sense of Definition \ref{def6.1}.
Moreover, $\bu$ satisfies the energy inequality
\begin{equation}\label{e6.45}
I_\lam(\bu(s)) +\int_0^s \norm{\bu_t(t)}{L^2}^2\, dt \leq
I_\lam(\bu_0)\qquad\mbox{for a.e. }  s\in [0,T],
\end{equation}
where
\begin{equation}\label{e6.46}
I_\lam(\bu):= |D \bu|(\Ome) +\frac{\lam}2\int_\Ome\, |\bu-\bg|^2\,dx .
\end{equation}
\end{theorem}

\begin{proof}
The proof is divided into four steps. 

\medskip
{\em Step 1: Extracting a convergent subsequence and passing to the limit}.
Let $\bu^\vepsi$ denote the solution of \reff{e4.1}-\reff{e4.4} constructed 
in Theorem \ref{thm4.1}.
It follows from \reff{e4.1}, \reff{e4.6a}, \reff{e4.7}, 
and \reff{e4.8} that $\{\bu^{\vepsi}\}_{\vepsi>0}$ is uniformly
bounded in $L^\infty((0,T);$ $[W^{1,1}(\Ome)\cap
L^\infty(\Ome)]^n) \cap H^1((0,T); L^2(\Ome,\mathbf{R}^n))$,
$\{\cB^{\vepsi}\}_{\vepsi>0}$ in
$L^\infty((0,T);L^\infty(\Ome,B_1(\mathbf{R}^{n\times m})))$ 
and $\{\Div\cB_1^\vepsi\}_{\vepsi>0}$ in 
$L^2((0,T);L^1(\Ome;\mathbf{R}^n))$, and
$\{\mathbf{\sigma}_{\lam}^{\vepsi}\}_{\vepsi>0}$ is uniformly
bounded in $L^\infty((0,T);$ $L^1(\Ome,\mathbf{R}^n))$. Since
$L^1(\Ome)\subset \mathcal{M}(\Ome)$ and $W^{1,1}(\Ome)\subset
BV(\Ome)$, by the weak compactness of $\mathcal{M}(\Ome)$ and
$BV(\Ome)$ (cf. \cite{AFP}) we have that there exist subsequences
of $\{\bu^{\vepsi}\}_{\vepsi>0}$, $\{\cB^{\vepsi}\}_{\vepsi>0}$,
$\{\cB_1^{\vepsi}\}_{\vepsi>0}$, and 
$\{\mathbf{\sigma}_{\lam}^{\vepsi}\}_{\vepsi>0}$ (still denoted
by the same notation), respectively, and maps $\bu \in
L^\infty((0,T);$ $ [BV(\Ome)\cap L^\infty(\Ome)]^n)$ $ \cap
H^1((0,T); L^2(\Ome,\mathbf{R}^n))$, $\mathcal{B} \in
L^\infty((0,T); L^\infty(\Ome,B_1(\mathbf{R}^{n\times m})))$, 
$\mathbf{\nu}\in L^2((0,T);[\mathcal{M}(\Ome)]^n)$, and
$\mathbf{\sigma}_{\lam} \in L^\infty((0,T);$ 
$[\mathcal{M}(\Ome)]^n)$, respectively, such that as
$\vepsi\rightarrow 0$
\begin{alignat}{2}\label{e6.47}
\bu^{\vepsi} &\longrightarrow \bu &&\qquad\mbox{weakly* in }
L^\infty((0,T);[BV(\Ome)\cap L^\infty(\Ome)]^n), \\
& &&\qquad\mbox{strongly in } L^2((0,T);L^1(\Ome,\mathbf{R}^n)) \label{e6.48}\\
& &&\qquad\mbox{a.e. in } \Ome_T, \label{e6.49} \\
b_1(\vepsi)\nab\bu^\vepsi &\longrightarrow 0 &&\qquad\mbox{weakly
in } L^2((0,T);L^2(\Ome,\mathbf{R}^{n\times m})), \label{e6.50}\\
\bu^{\vepsi}_t &\longrightarrow \bu_t  &&\qquad\mbox{weakly in }
L^2((0,T);L^2(\Ome,\mathbf{R}^n)), \label{e6.51}\\
\Div\cB_1^\vepsi &\longrightarrow \mathbf{\nu} &&\qquad\mbox{weakly* in }
L^2((0,T);[\mathcal{M}(\Ome)]^n) ,\label{e6.52}\\
\cB^\vepsi &\longrightarrow \cB &&\qquad\mbox{weakly* in }
L^\infty((0,T);L^\infty(\Ome,B_1(\mathbf{R}^{n\times m}))), \label{e6.53} \\
\mathbf{\sigma}_{\lam}^\vepsi &\longrightarrow
\mathbf{\sigma}_{\lam}
 &&\qquad\mbox{weakly* in }
L^\infty((0,T); [\mathcal{M}(\Ome)]^n). \label{e6.54}
\end{alignat}

It follows immediately from \reff{e6.49} and \reff{e4.15} that
\begin{equation}\label{e6.55}
|\bu| = 1\qquad\mbox{a.e. in } \Ome_T ,
\end{equation}
and an application of the Lebesgue dominated convergence theorem yields that
\begin{equation}\label{e6.56}
\bu^{\vepsi} \overset{\vepsi\searrow 0}{\longrightarrow}\bu
\qquad\mbox{strongly in } L^r((0,T);L^r(\Ome,\mathbf{R}^n)),\quad\forall
r\in [1, \infty).
\end{equation}

Now, taking $\vepsi\rightarrow 0$ in \reff{e4.21}, and
\reff{e4.25} (with $p=1$) we have for any $\bv\in
[C^1({\overline{\Ome}}_T)]^n$ and $\bw\in L^\infty((0,T);H^1(\Ome;\bR^n))$ 
\begin{eqnarray}\label{e6.57}
&&\int_0^T \Bigl\{ \bigl(\, \bu_t\wedge\bu, \bw\,\bigr)
+\bigl(\, \cB\wedge\bu,\nab\bw\,\bigr)
-\lam\,\bigl(\,\bg\wedge\bu,\bw\,\bigr)\Bigr\}\,dt =0, \\ 
&&\int_0^T \Bigl\{ \bigl(\, \bu_t, \bv\,\bigr) +\bigl(\,\cB, \nab\bv\,\bigr) +\lam
\bigl(\, \bu-\bg, \bv\,\bigr)\Bigr\}\,dt 
= \int_0^T \Bigl(\int_\Ome \bv d\mathbf{\sigma}_{\lam}\Bigr)\, dt .  \label{e6.58}
\end{eqnarray}

\medskip
{\em Step 2: Identifying $\mathbf{\nu}$ and $\mathbf{\sigma}_\lam$}.
Firstly, it follows from the identity $(\cB^\vepsi)^T\bu^\vepsi=0$ (cf.
\reff{e4.22}), \reff{e6.53}, and \reff{e6.56} that
\begin{equation}\label{e6.59}
\cB^T \bu = 0\qquad\mbox{a.e. in } \Ome_T .
\end{equation}

Secondly, for any $\bv\in [C^1(\overline{\Ome}_T)]^n$, it follows from \reff{e6.50},
\reff{e6.52}, and \reff{e6.53} that
\[
\int_{\Ome_T} \bv\cdot d\mathbf{\nu}=
\lim_{\vepsi\rightarrow 0} \int_{\Ome_T} \bv\cdot \Div\,\cB_1^\vepsi\,dx dt
=-\lim_{\vepsi\rightarrow 0}\int_{\Ome_T} \nab\bv\cdot\cB_1^\vepsi\,dx dt
=-\int_{\Ome_T} \nab\bv\cdot \cB\, dx dt.
\]
Hence, $\Div\cB$ exists and 
\begin{equation}\label{e6.60}
\Div\,\cB= \mathbf{\nu}, \qquad\mbox{therefore }\quad 
\cB\in L^2((0,T);\mathcal{DT}(\Ome)).
\end{equation}
It then follows from \eqref{e6.25} that 
\[
\cB \mathbf{n} =0 \qquad\text{on }  \p\Ome_T.
\]

Thirdly, notice that \reff{e6.57} immediately implies that 
\begin{equation}\label{e6.61}
\Div (\cB\wedge \bu) \in L^2(\Ome_T, \bR^n),\qquad\mbox{hence }\quad
\cB\wedge \bu \in L^2((0,T); \bY(\Ome)_2).
\end{equation}

Let $\{\bu_\rho\}$ denote the smooth approximation sequence of $\bu$ as
constructred in Theorem 3.9 of \cite{AFP}. For any $\bv\in [C^1_0(\Ome_T)]^n$, 
set $\bw=\bu_\rho\wedge\bv$ in \reff{e6.57} we get
\begin{equation}\label{e6.62}
\int_0^T \Bigl\{ \bigl(\bu_t, \bu\wedge(\bu_\rho\wedge \bv) \bigr)
+\bigl(\cB\wedge\bu,\nab(\bu_\rho\wedge \bv) \bigr)
-\lam\,\bigl(\bg, \bu\wedge(\bu_\rho\wedge \bv) \bigr)\Bigr\}\,dt=0.
\end{equation}
It follows from \reff{e6.55}, the convergence property of $\bu_\rho$
(cf. \cite{AFP}), and the identity $\bu\wedge (\bu\wedge \bv)
= \bu (\bu\cdot \bv)- \bv$ that
\begin{eqnarray*}
&&\lim_{\rho\rightarrow 0} \bigl(\bu_t, \bu\wedge(\bu_\rho\wedge \bv) \bigr)
= \bigl(\bu_t, \bu\wedge(\bu\wedge \bv) \bigr)
= \bigl(\bu_t, \bu (\bu\cdot \bv)-\bv) )
=- \bigl(\bu_t,\bv \bigr), \\
&& \lim_{\rho\rightarrow 0}\bigl(\bg, \bu\wedge(\bu_\rho\wedge \bv) \bigr)
= \bigl(\bg, \bu\wedge(\bu\wedge \bv) \bigr)
= \bigl(\bu-\bg, \bv\bigr) 
- \bigl((1-\bg\cdot\bu)\bu, \bv\bigr).
\end{eqnarray*}
It follows from Theorem 3.9 of \cite{AFP},  Lemma \ref{lem6.9}  and the identity
\[
\bigl(\cB\wedge\bu, \nab(\bu_\rho\wedge\bv)\bigr) 
=\bigl( (\cB\wedge\bu)\wedge\bu_\rho,\nab\bv\bigr) 
 + \bigl( (\cB\wedge\bu)\wedge\nab\bu_\rho,\bv\bigr)
\]
that
\begin{eqnarray*}
\lim_{\rho\rightarrow 0} \bigl(\cB\wedge\bu, \nab(\bu_\rho\wedge\bv)\bigr)
&=&\bigl( (\cB\wedge\bu)\wedge\bu,\nab\bv\bigr)
 + \langle (\cB\wedge\bu)\wedge D\bu,\bv \rangle \\
&=&-\bigl( \cB,\nab\bv\bigr)
 + \langle (\cB\wedge\bu)\wedge D\bu,\bv \rangle.
\end{eqnarray*}
Here we have used the fact that $(\cB\wedge\bu)\wedge\bu=-\cB$
in the light of \reff{e6.55} and \reff{e6.59}, and the measure $(\cB\wedge\bu)\wedge D\bu$
is defined by \reff{e6.14} with $\cA=\cB\wedge\bu$ and $\bv=\bu$.

Finally, substituting the above three equations into \reff{e6.62} and
multiplying the equation by $(-1)$ we get
\begin{eqnarray}\label{e6.100}
\int_0^T \Bigl\{ \bigl(\bu_t, \bv\bigr) +\bigl(\cB, \nab\bv\bigr) +\lam
\bigl(\bu-\bg, \bv\bigr) \Bigr\}\,dt
&&=\int_0^T \bigl\langle (\cB\wedge\bu)\wedge D\bu \\
&&\qquad
+\lam (1-\bg\cdot\bu)\bu, \bv\bigr\rangle \,dt \nonumber
\end{eqnarray}
for any $\bv\in [C^1_0(\Ome_T)]^n$. This and \reff{e6.58} imply that 
\[
\mathbf{\sigma}_{\lam}=(\cB\wedge\bu)\wedge D\bu+\lam (1-\bg\cdot\bu)\bu.
\]

\medskip
{\em Step 3: Identifying $\cB$}.
First, since $\{\cB^\vepsi\cdot\nab\bu^\vepsi\}$ is uniformly bounded
in $L^2((0,T);$ $L^1(\Ome))$, then there exists a subsequence (still denoted by the
same notation) and $\mu\in L^2((0,T);\cM(\Ome))$ such that
\begin{equation}\label{e6.63}
\cB^\vepsi\cdot\nab\bu^\vepsi\longrightarrow \mu
\quad\mbox{weakly* in } L^2((0,T);\cM(\Ome)).
\end{equation}

Let $\bu_\rho^\vepsi$ and $\bu_\rho$ denote mollified sequences for
$\bu^\vepsi$ and $\bu$, respectively, through a positive symmetric
mollifier. For any open set $E\subset \Ome$ we have
\begin{eqnarray}\label{e6.64}
\lim_{\vepsi\rightarrow 0}\int_0^T\int_E \cB^\vepsi\cdot\nab\bu^\vepsi\,dx dt
&=& \lim_{\vepsi\rightarrow 0} \lim_{\rho\rightarrow 0}
\int_0^T \int_E \cB^\vepsi\cdot D\bu^\vepsi_\rho \, dx dt \\
&=& \lim_{\rho\rightarrow 0} \lim_{\vepsi\rightarrow 0}
\int_0^T \int_E \cB^\vepsi\cdot D\bu^\vepsi_\rho\, dx dt \nonumber\\
&=& \lim_{\rho\rightarrow 0} \int_0^T\, (\cB\cdot D\bu_\rho)(E)\, dt
\quad \mbox{(by \reff{e6.53}, \reff{e6.56})} \nonumber\\
&=& \int_0^T\,\overline{\cB\cdot D\bu }\,(E)\,dt ,
\quad \mbox{(by \reff{e6.37}) } \nonumber
\end{eqnarray}
where $\overline{\cB\cdot D\bu }$ is defined in Lemma \ref{lem6.7}.

Hence, it follows from \reff{e6.63}, \reff{e6.64} and Lemma \ref{lem6.7} that
\begin{eqnarray}\label{e6.65}
\mu = \overline{\cB\cdot D\bu } << |D\bu| .
\end{eqnarray}
We refer the reader to Definition 1.24 of \cite{AFP} for the notation ``$<<$". 

On the other hand, a direct calculation yields that
\[
\cB^\vepsi\cdot \nab\bu^\vepsi = \frac{|\nab\bu^\vepsi|^2}{|\nab\bu^\vepsi|_\vepsi} 
\geq |\nab\bu^\vepsi|_\vepsi - \vepsi  .
\]
Setting $\vepsi\rightarrow 0$ and by the lower semicontinuity
of the $BV$-norm  (cf. \cite{AFP}) we obtain
\[
\mu >> |D \bu|,
\]
which together with \reff{e6.65} and Lemma \ref{lem6.7} yield that
\begin{equation}\label{e6.66}
|D\bu|=\mu=\overline{\cB\cdot D\bu},
\end{equation}
hence,
\begin{equation}\label{e6.67}
|D\bu|^a=(\overline{\cB\cdot D\bu})^a=\cB\cdot (D\bu)^a .
\end{equation}

Finally, we note that all other properties of $\cB$ and the
measure $(\cB\wedge\bu)\wedge D\bu$ listed in (vi) of Definition \ref{def6.1}
are immediate consequences of \reff{e6.66} and Lemmas \ref{lem6.2}, 
\ref{lem6.3}-\ref{lem6.7}
and Corollary \ref{cor6.1}.

\medskip
{\em Step 4: Finishing up.} 
We now conclude the proof of the theorem by showing the
energy inequality \reff{e6.45}. First, notice that \reff{e4.6a}
implies that for a.e. $s\in [0,T]$
\begin{eqnarray}\label{e6.68}
I_\lam (\bu^\vepsi(s)) + \int_0^s \norm{\bu^{\vepsi}_t(t)}{L^2}^2\, dt
&\leq& J_{1,\lambda}^\vepsi (\bu_0) \\
&\leq& \frac{b_1(\vepsi)}2\norm{\nab\bu_0}{L^2}^2 + a_1(\vepsi) |\Ome|
+ I_\lam (\bu_0). \nonumber
\end{eqnarray}
Then \reff{e6.45} follows from taking $\vepsi\rightarrow 0$ in
\reff{e6.68}, using the lower semicontinuity of the $BV$-seminorm
and the $L^2$-norm with respect to $L^2$-weak convergence.
The proof is complete.
\end{proof}

\begin{remark}\label{rem6.2}
(a). Since weak solutions to \reff{e1.8}-\reff{e1.11} are not
unique in general for $1<p<\infty$ (cf. \cite{Co,Hu1,M2}), we
expect that this nonuniqueness also holds for the case $p=1$.

(b). The existence of Theorem \ref{thm6.1} is proved under the assumption
$\bu_0\in H^1(\Ome,\mathbf{R}^n)$. This assumption can be weaken to
$\bu_0\in [BV(\Ome)\cap L^\infty(\Ome)]^n$ using a smoothing technique.
\end{remark}

\section{Fully discrete finite element approximations} \label{sec-7}

\subsection{Formulation of fully discrete finite element methods}\label{sec-7.1}
For ease of exposition, we assume $\Ome$ is a polytope in this section.
Let $\Tc_h$ be a quasi-uniform ``triangulation" of the domain
$\Ome$ of mesh size $0<h<1$ and $\overline{\Omega} = \bigcup_{K
\in \mathcal{T}_h} \overline{K}$ ($K\in \mathcal{T}_h$ are
tetrahedrons in the case $m=3$). Let $J_\tau:=\{t_k\}_{k=0}^L$ be a
uniform partition of $[0,T]$ with mesh size $\tau:=\frac{T}{L}$,
and $\p_t \bv^k:=(\bv^k-\bv^{k-1})/\tau$. For an integer
$r\geq 1$, let $P_r(K)$ denote the space of polynomials of degree less
than or equal to $r$ on $K$. We introduce the finite element space
\[
\bV^h=\bigl\{ \bv_h\in C(\overline{\Ome},\bR^n)\cap H^1(\Ome,\bR^n);\,
\bv_h|_K\in [P_r(K)]^n,\,\forall K\in\cT_h\bigr\} .
\]

Notice that the density function $F$ defined in \reff{e1.15} is
not a convex function. On the other hand, there exist two convex
functions $W_{+}$ and $W_{-}$ such that
\begin{equation}\label{e7.1a}
F(\bv)=W_{+}(\bv)-W_{-}(\bv).
\end{equation}
One such an example is $W_{+}(\bv)=\frac{|\bv|^4}4$ and
$W_{-}(\bv)=\frac{|\bv|^2}2-\frac14$. Clearly, the above
decomposition is not unique.

We are now ready to introduce our fully discrete finite element
discretizations for the initial boundary value problem
\reff{e1.18}-\reff{e1.20}. Find $\bu^k_h\in \bV^h$ for
$k=1,2,\cdots,L$ such that
\begin{eqnarray}\label{e7.2}
&&\bigl(\p_t\bu_h^k, \bv_h\bigr) + \bigl(\cB_h^k, \nab\bv_h\bigr)
+\lam\bigl(\bu_h^k-\bg,\bv_h\bigr) \\
&&\hskip 1.5in
+\frac{1}{\del}\bigl(\, W_{+}^\prime(\bu_h^k),\bv_h\,\bigr)
=\frac{1}{\del}\bigl(W_{-}^\prime(\bu_h^{k-1}),\bv_h\bigr)
\quad \forall \bv\in\bV^h, \nonumber\\
&&\cB_h^k=\bigl[\,b_p(\vepsi) + |\nab \bu_h^k|_\vepsi^{p-2}\,\bigr]\nab\bu_h^k
\label{e7.3}
\end{eqnarray}
with some starting value $\bu_h^0\in \bV^h$ to be specified later.
Note that for notational brevity we have omitted the indices $\vepsi,\del$
and $p$ on $\bu_h^k$ and $\cB_h^k$

For each $k$, equation \reff{e7.2} is a nonlinear equation in
$\bu_h^k$. Hence, the above numerical method is an implicit
scheme, its well-posedness is ensured by the following theorem.

\begin{theorem}\label{thm7.1}
For each fixed $k\geq 1$, suppose that $\bu_h^{k-1}\in \bV^h$ is known, then there
exists a unique solution $\bu_h^k \in \bV^h$ to \reff{e7.2}-\reff{e7.3}. Moreover,
$\{\bu_h^k\}_{k=0}^L$ satisfies the following energy estimate
\begin{eqnarray}\label{e7.5}
\frac{\tau}2\sum_{k=1}^\ell\, \norm{\p_t\bu_h^k}{L^2}^2 +
J_{p,\lam}^{\vepsi,\del}(\bu_h^k) \leq
J_{p,\lam}^{\vepsi,\del}(\bu_h^0) 
\qquad \mbox{for } 1 \leq \ell \leq L.
\end{eqnarray}
Here $J_{p,\lam}^{\vepsi,\del}$ is defined by \reff{e1.17}.
\end{theorem}

\begin{proof}
For each fixed $k\geq 1$, it is easy to check that \reff{e7.2}-\reff{e7.3} is
the Euler-Lagrange equation of the following functional over $\bV^h$
\begin{align}\label{e7.6}
G_k(\bv)&:=\int_\Ome\Bigl\{ \frac{1}{2\tau} |\bv-\bu_h^{k-1}|^2
+\frac{b_p(\vepsi)}2\, |\nab\bv|^2 +\frac{1}{p}\, |\nab
\bv|_\vepsi^p +\frac{\lam}2|\bv-\bg|^2 \\
&\hskip 1.4in
+\frac{1}{\del} W_{+}(\bv) \,\Bigr\}\,dx -\frac{1}{\del}\,
\int_\Ome\,W_{-}^\prime(\bu_h^{k-1})\cdot\bv\, dx. \nonumber
\end{align}
Since $G_k$ is a convex, coercive and differentiable functional,
then it has a unique minimizer $\bu_h^k\in\bV^h$ (cf. \cite{S3}),
hence, \reff{e7.2}-\reff{e7.3} has a unique solution.

Since $\bu_h^k$ is the minimizer of $G_k$ over $\bV^h$, we have that
\begin{equation}\label{e7.7}
G_k(\bu_h^k)\leq G_k(\bu_h^{k-1}).
\end{equation}
It follows from the convexity of $W_{-}$ that
\[
W_{-}^\prime(\bu^{k-1}_h)\bigl( \bu^k_h-\bu^{k-1}_h \bigr) \leq
W_{-}(\bu^k_h)-W_{-}(\bu^{k-1}_h) .
\]
This and \reff{e7.7} imply that
\[
\frac{1}2\norm{\p_t \bu_h^k}{L^2}^2 
+ \frac{ J_{p,\lam}^{\vepsi,\del}(\bu_h^k)- J_{p,\lam}^{\vepsi,\del}(\bu_h^{k-1})}{\tau} \leq 0.
\]
The bound \reff{e7.5} then follows from applying the summation
operator $\tau\sum_{k=1}^\ell\, (1\leq \ell \leq L)$ to the last
inequality. Hence the proof is complete.
\end{proof}

\subsection{Convergence analysis}\label{sec-7.2}
The goal of this subsection is to show that the numerical solution
of \reff{e7.2}-\reff{e7.3} converges to the unique weak solution of
\reff{e1.18}-\reff{e1.20} as $h, \tau\rightarrow 0$. There are two
approaches to reach this goal. The first approach assumes the
existence of the solution of \reff{e1.18}-\reff{e1.20}, which in
fact has been proved in Theorem \ref{thm3.2}, and then proves
that $\bu_h^k$ converges to that solution. The other approach shows the
convergence {\em without} assuming the existence of the solution
of \reff{e1.18}-\reff{e1.20}. This can be done by applying the
energy method and compactness argument used in the proof of
Theorem \ref{thm3.2} to the finite element solution $\{\bu_h^k\}$.
In the following, we shall go with the latter approach since this
will also provide an alternative proof for Theorem \ref{thm3.2} as
alluded to in Remark \ref{rem3.1} (c).

For the fully discrete finite element solution $\{\bu_h^k\}$, we
define its linear interpolation in $t$ as follows
\begin{equation}
\bU^{\vepsi,\del,h,\tau}(\cdot,t):= \frac{t-t_{k-1}}{\tau} \bu_h^{k}(\cdot)
+ \frac{t_{k}-t}{\tau} \bu_h^{k-1}(\cdot)\quad \forall \, t\in [t_{k-1},t_{k}],
\quad 1\leq k\leq L.  \label{e7.10}
\end{equation}
Clearly, $\bU^{\vepsi,\del,h,\tau}$ is continuous in both $x$ and $t$.

The main result of this section is the following convergence theorem.

\begin{theorem}\label{thm7.2}
For $1\leq p<\infty$, suppose that $\bu_0\in 
W^{1,p^*}(\Ome, \mathbf{R}^n)$, $|\bu_0|= 1$ 
and $|\bg|\leq 1$ in $\Ome$.  For each pair of positive 
numbers $(\vepsi,\del)$, let $\bU^{\vepsi,\del,h,\tau}$ be defined 
by \reff{e7.10}. Then, there exists $\bu^{\vepsi,\del}\in L^\infty(\Ome_T)$ 
such that
\begin{equation}\label{e7.8}
\lim_{h,\tau\rightarrow 0}\,
\norm{\bu^{\vepsi,\del}-\bU^{\vepsi,\del,h,\tau}} {L^q(\Ome_T)} =
0\qquad \forall q\in [1,\infty),
\end{equation}
provided that 
\[
\lim_{h\rightarrow 0}\ \norm{\bu_0-\bu_h^0}{W^{1,p^*}(\Ome)} =0.
\]
Moreover, $\bu^{\vepsi,\del}$ solves \reff{e1.18}-\reff{e1.20}
in the sense of Definition \ref{def3.1}.
\end{theorem}

\begin{proof}
The proof follows the same lines as that of Theorem 1.5 of \cite{CHH},
where the convergence of a general Galerkin approximation was proved for the
case $p\geq 2$. Since the finite element approximation is a special
Galerkin approximation, the proof of Theorem 1.5 of \cite{CHH}
can easily be adapted to the finite element approximation $\bU^{\vepsi,\del,h,\tau}$
for $p\geq 2$ thanks to the discrete energy estimate \reff{e7.5} and
the facts that $\p_t\bu_h^k=\bU^{\vepsi,\del,h,\tau}_t$ and
$\tau\sum_{k=1}^L \norm{\p_t\bu_h^k}{L^2}^2=\norm{\bU^{\vepsi,\del,h,\tau}_t}
{L^2(L^2)}$.

Since the operator $-\Del^\vepsi_p$ is {\em uniformly} elliptic, as a result,
the compactness of Lemma \ref{lem2.2} not only holds for $p\geq 2$ but
also for $1\leq p<2$. In addition, note that $\bU^{\vepsi,\del,h,\tau}$ is
uniformly (in $h$ and $\tau$) bounded in $L^\infty((0,T);H^1(\Ome,\bR^n))$ for
$1\leq p<2$. Hence, the proof of Theorem 1.5 of \cite{CHH} can be
adapted with slight modifications to prove \reff{e7.8} for the case
$1\leq p<2$.
\end{proof}

\begin{remark}
Several practical choices of $\bu_h^0$ are possible. For instance,
both the $L^2$-projection of $\bu_0$ and the Clem\'ent finite
element interpolation of $\bu_0$ into $\bV^h$ (cf. \cite{Ci}) are
qualified candidates for $\bu_h^0$.
\end{remark}

An immediate consequence of Theorems \ref{thm4.1}, \ref{thm5.1}, 
\ref{thm6.1}, \ref{thm7.1}, and \ref{thm7.2} is the following convergence theorem.

\begin{theorem}\label{thm7.3}
Let $1\leq p<\infty$ and $\bU^{\vepsi,\del,h,\tau}$ be defined 
by \reff{e7.10}, assume the assumptions of Theorems \ref{thm7.1}, \ref{thm7.2}, 
\ref{thm4.1}, \ref{thm5.1} and
\ref{thm6.1} hold. Then, there exists a subsequence of 
$\{\bU^{\vepsi,\del,h,\tau}\}$ (still denoted by the same notation)
and a weak solution $\bu$ of \reff{e1.8}-\reff{e1.11} such that 
\begin{equation}\label{e7.9}
\lim_{\vepsi,\delta\rightarrow 0,}\,\lim_{h,\tau\rightarrow 0,}\, 
\norm{\bu-\bU^{\vepsi,\del,h,\tau}} {L^q(\Ome_T)} =
0\qquad \forall q\in [1,\infty).
\end{equation}
\end{theorem}

\vskip .2in
{\bf Acknowledgment:} 
The authors would like to thank the Mathematisches \hfill\newline 
Forschungsinstitut 
Oberwolfach for the kind hospitality and opportunity of its ``Research in Pairs" 
program. The second author would also like to thank Professor Fanghua Lin
for a helpful discussion, and to thank the Institute of Mathematics and its 
Applications (IMA) of University of Minnesota for its support and hospitality 
during the author's recent visit to the IMA.


\end{document}